\documentclass[a4paper,11pt]{article}
\usepackage{amsmath,amsthm,amssymb}
\usepackage[mathscr]{eucal}
\usepackage[bookmarks=false]{hyperref}

{\theoremstyle{definition}\newtheorem{definition}{Definition}

\newtheorem{remark}[definition]{Remark}
\newtheorem{example}[definition]{Example}}

\newtheorem{proposition}[definition]{Proposition}
\newtheorem{lemma}[definition]{Lemma}
\newtheorem{theorem}[definition]{Theorem}

\newcommand{\Inn}{\operatorname{Inn}}

\newcommand{\C}{\mathbb{C}}

\newcommand{\cR}{\mathcal{R}}

\newcommand{\actson}{\curvearrowright}
\newcommand{\SL}{\operatorname{SL}}
\newcommand{\rL}{\mathord{\text{\rm L}}}
\newcommand{\Aut}{\operatorname{Aut}}
\newcommand{\Out}{\operatorname{Out}}
\newcommand{\N}{\mathbb{N}}
\newcommand{\T}{\mathbb{T}}
\newcommand{\Z}{\mathbb{Z}}
\newcommand{\cF}{\mathcal{F}}

\newcommand{\cV}{\mathcal{V}}
\newcommand{\id}{\mathord{\operatorname{id}}}
\newcommand{\si}{\sigma}

\newcommand{\recht}{\rightarrow}
\newcommand{\cU}{\mathcal{U}}
\newcommand{\vphi}{\varphi}
\newcommand{\cW}{\mathcal{W}}
\newcommand{\R}{\mathbb{R}}
\newcommand{\al}{\alpha}
\newcommand{\eps}{\varepsilon}

\newcommand{\Tr}{\operatorname{Tr}}
\newcommand{\ovt}{\mathbin{\overline{\otimes}}}
\newcommand{\B}{\operatorname{B}}
\newcommand{\om}{\omega}
\newcommand{\cP}{\mathcal{P}}

\newcommand{\Q}{\mathbb{Q}}
\newcommand{\Ker}{\operatorname{Ker}}

\newcommand{\ot}{\otimes}

\newcommand{\Ad}{\operatorname{Ad}}
\newcommand{\cG}{\mathcal{G}}

\newcommand{\Stab}{\operatorname{Stab}}

\newcommand{\PSL}{\operatorname{PSL}}

\newcommand{\cN}{\mathcal{N}}

\newcommand{\Om}{\Omega}

\newcommand{\cC}{\mathcal{C}}
\newcommand{\Deltab}{\overline{\Delta}}
\newcommand{\Thetab}{\overline{\Theta}}
\newcommand{\Psib}{\overline{\Psi}}
\newcommand{\deltab}{\overline{\delta}}

\newcommand{\etab}{\overline{\eta}}
\newcommand{\cartan}{\operatorname{Cartan}}
\newcommand{\vnalg}{\operatorname{vNalg}}
\newcommand{\masa}{\operatorname{masa}}
\newcommand{\subgr}{\operatorname{subgr}}
\newcommand{\trees}{\operatorname{Trees}}

\begin{document}

\begin{center}
{\boldmath\LARGE\bf A class of II$_1$ factors with many non conjugate \vspace{0.5ex}\\ Cartan subalgebras}

\bigskip

{\sc by An Speelman\footnote{K.U.Leuven, Department of Mathematics, an.speelman@wis.kuleuven.be \\ Research Assistant of the Research Foundation --
    Flanders (FWO)} and Stefaan Vaes\footnote{K.U.Leuven, Department of Mathematics, stefaan.vaes@wis.kuleuven.be \\ Partially
    supported by ERC Starting Grant VNALG-200749, Research
    Programme G.0639.11 of the Research Foundation --
    Flanders (FWO) and K.U.Leuven BOF research grant OT/08/032.}}
\end{center}

\begin{abstract}\noindent
We construct a class of II$_1$ factors $M$ that admit unclassifiably many Cartan subalgebras in the sense that the equivalence relation of being conjugate by an automorphism of $M$ is complete analytic, in particular non Borel. We also construct a II$_1$ factor that admits uncountably many non isomorphic group measure space decompositions, all involving the same group $G$. So $G$ is a group that admits uncountably many non stably orbit equivalent actions whose crossed product II$_1$ factors are all isomorphic.
\end{abstract}

\section{Introduction}

Free ergodic probability measure preserving (p.m.p.) actions $\Gamma \actson (X,\mu)$ of countable groups give rise to II$_1$ factors $\rL^\infty(X) \rtimes \Gamma$ by Murray and von Neumann's group measure space construction. In \cite{Si55} it was shown that the isomorphism class of $\rL^\infty(X) \rtimes \Gamma$ only depends on the orbit equivalence relation on $(X,\mu)$ given by $x \sim y$ iff $x \in \Gamma \cdot y$. More precisely, two free ergodic p.m.p.\ actions $\Gamma \actson X$ and $\Lambda \actson Y$ are orbit equivalent iff there exists an isomorphism of $\rL^\infty(X) \rtimes \Gamma$ onto $\rL^\infty(Y) \rtimes \Lambda$ sending $\rL^\infty(X)$ onto $\rL^\infty(Y)$.

If $\Gamma \actson (X,\mu)$ is essentially free and $M := \rL^\infty(X) \rtimes \Gamma$, then $A := \rL^\infty(X)$ is a maximal abelian subalgebra of $M$ that is regular: the group of unitaries $u \in \cU(M)$ that normalize $A$ in the sense that $u A u^* = A$, spans a dense $*$-subalgebra of $M$. A regular maximal abelian subalgebra of a II$_1$ factor is called a \emph{Cartan subalgebra.}

In order to classify classes of group measure space II$_1$ factors $M := \rL^\infty(X) \rtimes \Gamma$ in terms of the group action $\Gamma \actson (X,\mu)$, one must solve two different problems. First one should understand whether the Cartan subalgebra $\rL^\infty(X) \subset M$ is, in some sense, unique. If this is the case, any other group measure space decomposition $M = \rL^\infty(Y) \rtimes \Lambda$ must come from an orbit equivalence between $\Gamma \actson X$ and $\Lambda \actson Y$. The second problem is then to classify the corresponding class of orbit equivalence relations $\cR(\Gamma \actson X)$ in terms of the group action. Both problems are notoriously hard, but a huge progress has been made over the last 10 years thanks to Sorin Popa's deformation/rigidity theory (see \cite{Po06a,Ga10,Va10a} for surveys). In particular, many uniqueness results for Cartan subalgebras and group measure space Cartan subalgebras have been established recently in \cite{OP07,OP08,Pe09,PV09,FV10,CP10,HPV10,Va10b,Io11,CS11}.

In this paper we concentrate on ``negative'' solutions to the first problem by exhibiting a class of II$_1$ factors with many Cartan subalgebras. First recall that two Cartan subalgebras $A \subset M$ and $B \subset M$ of a II$_1$ factor are called \emph{unitarily conjugate} if there exists a unitary $u \in \cU(M)$ such that $B = u A u^*$. The Cartan subalgebras are called \emph{conjugate by an automorphism} if there exists $\al \in \Aut M$ such that $B = \al(A)$. In \cite{CFW81} it is shown that the hyperfinite II$_1$ factor $R$ has a unique Cartan subalgebra up to conjugacy by an automorphism of $R$. By \cite[Theorem 7]{FM75} (see \cite[Section 4]{Pa83} for details) the hyperfinite II$_1$ factor however admits uncountably many Cartan subalgebras that are not unitarily conjugate. In \cite{CJ82} Connes and Jones constructed the first example of a II$_1$ factor that admits two Cartan subalgebras that are not conjugate by an automorphism. More explicit examples of this phenomenon were found in \cite[Section 7]{OP08} and \cite[Example 5.8]{PV09}.

In \cite[Section 6.1]{Po06b} Popa discovered a free ergodic p.m.p.\ action $\Gamma \actson (X,\mu)$ such that the orbit equivalence relation $\cR(\Gamma \actson X)$ has trivial fundamental group, while the II$_1$ factor $M:=\rL^\infty(X) \rtimes \Gamma$ has fundamental group $\R_+$. Associated with this is a one parameter family $A_t$ of Cartan subalgebras of $M$ that are not conjugate by an automorphism of $M$. Indeed, put $A := \rL^\infty(X)$ and for every $t \in (0,1]$, choose a projection $p_t \in A$ of trace $t$ and a $*$-isomorphism $\pi_t : p_t M p_t \recht M$. Put $A_t := \pi_t(A p_t)$. Then $A_t \subset M$, $t \in (0,1]$, are Cartan subalgebras that are not conjugate by an automorphism.
Note that earlier (see \cite[Corollary of Theorem 3]{Po90} and \cite[Corollary 4.7.2]{Po86}) Popa found free ergodic p.m.p.\ actions $\Gamma \actson (X,\mu)$ such that the equivalence relation $\cR(\Gamma \actson X)$ has countable fundamental group, while the II$_1$ factor $\rL^\infty(X) \rtimes \Gamma$ has fundamental group $\R_+$. By the same construction, there are uncountably many $t$ such that the Cartan subalgebras $A_t$ are not conjugate by an automorphism. In all these cases, although the Cartan subalgebras $A_t \subset M$ are not conjugate by an automorphism of $M$, they are by construction conjugate by a stable automorphism of $M$. More precisely, we say that two Cartan subalgebras $A \subset M$ and $B \subset M$ are \emph{conjugate by a stable automorphism of $M$} if there exist non zero projections $p \in A$, $q \in B$ and a $*$-isomorphism $\al : pMp \recht qMq$ satisfying $\al(A p) = B q$.

In this paper we provide the first class of II$_1$ factors $M$ that admit uncountably many Cartan subalgebras that are not conjugate by a stable automorphism of $M$. Actually we obtain II$_1$ factors $M$ such that the equivalence relation ``being conjugate by an automorphism'' on the set of Cartan subalgebras of $M$ is complete analytic, in particular non Borel. In this precise sense the decision problem whether two Cartan subalgebras of these II$_1$ factors $M$ are conjugate by an automorphism, is as hard as it can possibly get (see Remark \ref{rem.descriptive}). Using Popa's conjugacy criterion for Cartan subalgebras \cite[Theorem A.1]{Po01}, we show that for any II$_1$ factor with separable predual, the set of Cartan subalgebras is a standard Borel space and the equivalence relation ``being unitarily conjugate'' is always Borel. We provide examples of II$_1$ factors where this equivalence relation is not concretely classifiable.

The ``many'' Cartan subalgebras in the II$_1$ factors $M$ in the previous paragraph are not of group measure space type. In the final section of the paper we give the first example of a II$_1$ factor that admits uncountably many Cartan subalgebras that are not conjugate by a (stable) automorphism and that all arise from a group measure space construction. They actually all arise as the crossed product with a single group $G$. This means that we construct a group $G$ that admits uncountably many free ergodic p.m.p.\ actions $G \actson (X_i,\mu_i)$ that are non stably orbit equivalent, but such that the II$_1$ factors $\rL^\infty(X_i) \rtimes G$ are all isomorphic. This group $G$ should be opposed to the classes of groups considered in \cite{PV09,FV10,CP10,HPV10,Va10b} for which isomorphism of the group measure space II$_1$ factors always implies orbit equivalence of the group actions.

\subsection*{Terminology}

A \emph{Cartan subalgebra} $A$ of a II$_1$ factor $M$ is a maximal abelian von Neumann subalgebra whose normalizer
$$\cN_M(A) := \{u \in \cU(M) \mid u A u^* = A \}$$
generates $M$ as a von Neumann algebra. As mentioned above, we say that two Cartan subalgebras $A \subset M$ and $B \subset M$ are \emph{unitarily conjugate} if there exists a unitary $u \in \cU(M)$ such that $uAu^* = B$. We say that they are \emph{conjugate by an automorphism of $M$} if there exists $\al \in \Aut M$ such that $\al(A)= B$. Finally $A$ and $B$ are said to be \emph{conjugate by a stable automorphism of $M$} if there exist non zero projections $p \in A$, $q \in B$ and a $*$-isomorphism $\al : pMp \recht qMq$ satisfying $\al(A p) = Bq$.

By \cite{FM75} every II$_1$ equivalence relation $\cR$ on $(X,\mu)$ gives rise to a II$_1$ factor $\rL \cR$ that contains $\rL^\infty(X)$ as a Cartan subalgebra. Conversely, if $\rL^\infty(X) = A \subset M$ is an arbitrary Cartan subalgebra of a II$_1$ factor $M$, we get a natural II$_1$ equivalence relation $\cR$ on $(X,\mu)$ and a measurable $2$-cocycle on $\cR$ with values in $\T$ such that $A \subset M$ is canonically isomorphic with $\rL^\infty(X) \subset \rL_\Omega \cR$.

If $\cR_i$ are II$_1$ equivalence relations on the standard probability spaces $(X_i,\mu_i)$, we say that $\cR_1$ and $\cR_2$ are \emph{orbit equivalent} if there exists a measure space isomorphism $\Delta : X_1 \recht X_2$ such that $(x,y) \in \cR_1$ iff $(\Delta(x),\Delta(y)) \in \cR_2$ almost everywhere. More generally, $\cR_1$ and $\cR_2$ are called \emph{stably orbit equivalent} if there exist non negligible measurable subsets $\cU_i \subset X_i$ such that the restricted equivalence relations $\cR_{i|\cU_i}$ are orbit equivalent.

Throughout this paper compact groups are supposed to be second countable. Two subgroups $K_1,K_2 < K$ of the same group $K$ are called \emph{commensurate} if $K_1 \cap K_2$ has finite index in both $K_1$ and $K_2$.

\subsection*{Statement of the main results}

Let $\Gamma$ be a countable group and $K$ a compact abelian group with countable dense subgroup $\Lambda < K$.  Consider the II$_1$ factor
$$M := \rL^\infty(K^\Gamma) \rtimes (\Gamma \times \Lambda)$$
where $\Gamma \actson K^\Gamma$ is the Bernoulli action given by $(g \cdot x)_h = x_{hg}$ and $\Lambda \actson K^\Gamma$ acts by diagonal translation: $(\lambda \cdot x)_h = \lambda x_h$.

Whenever $K_1 < K$ is a closed subgroup such that $\Lambda_1 := \Lambda \cap K_1$ is dense in $K_1$, denote
$$\cC(K_1) := \rL^\infty(K^\Gamma/K_1) \rtimes \Lambda_1 = \rL^\infty(K^\Gamma/K_1) \ovt \rL \Lambda_1 \; .$$
Note that $\cC(K_1)$ is an abelian von Neumann subalgebra of $M$. In Lemma \ref{lemma.cartan} we prove that $\cC(K_1) \subset M$ is always a Cartan subalgebra. As we will see there, the density of $\Lambda \cap K_1$ in $K_1$ is equivalent with $\cC(K_1)$ being maximal abelian in $M$.

The following is our main result, providing criteria for when the Cartan subalgebras $\cC(K_1)$ are unitarily conjugate, conjugate by a (stable) automorphism or inducing (stable) orbit equivalent relations. As we shall see later, concrete choices of $\Lambda \subset K$ then lead to examples of II$_1$ factors where the equivalence relation of ``being conjugate by an automorphism'' is complete analytic and where the equivalence relation of ``being unitarily conjugate'' is not concretely classifiable.

\begin{theorem} \label{thm.main}
Let $\Gamma$ be an icc property (T) group with $[\Gamma,\Gamma]=\Gamma$. Let $K$ be a compact abelian group with countable dense subgroup $\Lambda$. As above, write $M = \rL^\infty(K^\Gamma) \rtimes (\Gamma \times \Lambda)$. Take closed subgroups $K_i < K$, $i=1,2$, such that $\Lambda \cap K_i$ is dense in $K_i$.
\begin{enumerate}
\item The following statements are equivalent.
\begin{itemize}
\item The Cartan subalgebras $\cC(K_1)$ and $\cC(K_2)$ are unitarily conjugate.
\item The subgroups $K_1, K_2 < K$ are commensurate.
\item $\Lambda K_1 = \Lambda K_2$.
\end{itemize}

\item The following statements are equivalent.
\begin{itemize}
\item The Cartan subalgebras $\cC(K_1)$ and $\cC(K_2)$ are conjugate by an automorphism of $M$.
\item The Cartan subalgebras $\cC(K_1)$ and $\cC(K_2)$ are stably conjugate by an automorphism.
\item There exists a continuous automorphism $\delta \in \Aut K$ such that $\delta(\Lambda) = \Lambda$ and $\delta(\Lambda K_1) = \Lambda K_2$.
\end{itemize}

\item The following statements are equivalent.
\begin{itemize}
\item The II$_1$ equivalence relations associated with $\cC(K_1)$ and $\cC(K_2)$ are orbit equivalent.
\item The II$_1$ equivalence relations associated with $\cC(K_1)$ and $\cC(K_2)$ are stably orbit equivalent.
\item There exists a continuous automorphism $\delta \in \Aut K$ such that $\delta(\Lambda K_1) = \Lambda K_2$.
\end{itemize}
\end{enumerate}
\end{theorem}

In Section \ref{sec.complete-analytic} we shall see that for any II$_1$ factor $M$ with separable predual, the set $\cartan(M)$ of its Cartan subalgebras is a standard Borel space and that moreover the equivalence relation of ``being unitarily conjugate'' is Borel, while the equivalence relation of ``being conjugate by an automorphism'' is analytic (see Proposition \ref{prop.borel}). As the following examples show, the first can be non concretely classifiable, while the second can be complete analytic, hence non Borel.

Whenever $\Lambda \subset K$ is a compact group with countable dense subgroup, we call $K_1 \subset K$ an \emph{admissible} subgroup if $K_1$ is closed and $K_1 \cap \Lambda$ is dense in $K_1$.

\begin{theorem}\label{thm.ex}
Let $\Gamma$ be any icc property (T) group with $[\Gamma,\Gamma]=\Gamma$, e.g.\ $\Gamma = \SL(3,\Z)$.
\begin{itemize}
\item Whenever $\cP$ is a set of distinct prime numbers, denote $K := \prod_{p \in \cP} \Z/p\Z$ and consider the dense subgroup $\Lambda := \bigoplus_{p \in \cP} \Z/p\Z$. For every subset $\cP_1 \subset \cP$ the subgroup of $K$ given by $K(\cP_1) := \prod_{p \in \cP_1} \Z / p \Z$ is admissible and gives rise to a Cartan subalgebra $\cC(\cP_1)$ of $M := \rL^\infty(K^\Gamma) \rtimes (\Gamma \times \Lambda)$. Two such Cartan subalgebras $\cC(\cP_1)$ and $\cC(\cP_2)$ or their associated equivalence relations are conjugate in any of the possible senses if and only if the symmetric difference $\cP_1 \bigtriangleup \cP_2$ is finite.

    In particular the equivalence relations on $\cartan(M)$ given by ``being unitarily conjugate'' or ``being conjugate by a (stable) automorphism'' are not concretely classifiable.

\item Denote by $\Lambda := \Q^{(\infty)}$ the countably infinite direct sum of copies of the additive group $\Q$. Then $\Lambda$ admits a second countable compactification $K$ such that for the associated II$_1$ factor $M := \rL^\infty(K^\Gamma) \rtimes (\Gamma \times \Lambda)$ the equivalence relation of ``being conjugate by an automorphism of $M$'' is a complete analytic equivalence relation on $\cartan(M)$.
\end{itemize}
\end{theorem}

In Theorem \ref{thm.main} the second set of statements is, at least formally, stronger than the third set of statements. Example \ref{ex.cute} below shows that this difference really occurs. So in certain cases the Cartan subalgebras $\cC(K_1)$ and $\cC(K_2)$ give rise to isomorphic equivalence relations, but are nevertheless non conjugate by an automorphism of $M$. The reason for this is the following~: the Cartan inclusions $\cC(K_i) \subset M$ come with $2$-cocycles $\Om_i$~; although the associated equivalence relations are isomorphic, this isomorphism does not map $\Omega_1$ onto $\Omega_2$.

We finally observe that if $K_1 < K$ is an infinite subgroup, the Cartan subalgebra $\cC(K_1) \subset M$ is never of group measure space type. The induced equivalence relation is generated by a free action of a countable group (see Lemma \ref{lemma.cartan}), but the $2$-cocycle is non trivial (see Remark \ref{rem.McDuff}).

\begin{example} \label{ex.cute}
Choose elements $\al,\beta,\gamma \in \T$ that are rationally independent (i.e.\ generate a copy of $\Z^3 \subset \T$). Consider $K = \T^3$ and define the countable dense subgroup $\Lambda < K$ generated by $(\al,1,\al)$, $(1,\al,\beta)$ and $(1,1,\gamma)$. Put $K_1 := \T \times \{1\} \times \T$ and $K_2 := \{1\} \times \T \times \T$. Then, $\Lambda \cap K_1$ is generated by $(\al,1,\al)$, $(1,1,\gamma)$ and hence dense in $K_1$. Similarly $\Lambda \cap K_2$ is dense in $K_2$.

We have $\Lambda K_1 = \T \times \al^\Z \times \T$ and $\Lambda K_2 = \al^\Z \times \T \times \T$. Hence the automorphism $\delta(x,y,z) = (y,x,z)$ of $\T^3$ maps $\Lambda K_1$ onto $\Lambda K_2$. It is however easy to check that there exists no automorphism $\delta$ of $K$ such that $\delta(\Lambda K_1) = \Lambda K_2$ and $\delta(\Lambda) = \Lambda$.

So, by Theorem \ref{thm.main}, the Cartan subalgebras $\cC(K_1)$ and $\cC(K_2)$ of $M$ give rise to orbit equivalent relations, but are not conjugate by a (stable) automorphism of $M$.
\end{example}

In the final section of the paper we give an example of a group $G$ that admits uncountably many non stably orbit equivalent actions that all give rise to the same II$_1$ factor. The precise statement goes as follows and will be an immediate consequence of the more concrete, but involved, Proposition \ref{prop.concrete}.

\begin{theorem}\label{thm.gmsc}
There exists a group $G$ that admits uncountably many free ergodic p.m.p.\ actions $G \actson (X_i,\mu_i)$ that are non stably orbit equivalent but with all the II$_1$ factors $\rL^\infty(X_i) \rtimes G$ being isomorphic with the same II$_1$ factor $M$. So this II$_1$ factor $M$ admits uncountably many group measure space Cartan subalgebras that are non conjugate by a (stable) automorphism of $M$.
\end{theorem}

\section{Proof of the main Theorem \ref{thm.main}}

We first prove the following elementary lemma that establishes the essential freeness of certain actions on compact groups.

\begin{lemma}\label{lemma.free}
Let $Y$ be a compact group and $\Lambda < Y$ a countable subgroup. Assume that a countable group $\Gamma$ acts on $Y$ by continuous group automorphisms $(\al_g)_{g \in \Gamma}$ preserving $\Lambda$. Assume that for $g \neq e$ the subgroup $\{y \in Y \mid \al_g(y) = y\}$ has infinite index in $Y$. Equip $Y$ with its Haar measure.

Then the action of the semidirect product $\Lambda \rtimes \Gamma$ on $Y$ given by $(\lambda,g) \cdot y = \lambda \al_g(y)$ is essentially free.
\end{lemma}
\begin{proof}
Take $(\lambda,g) \neq e$ and consider the set of $y$ with $(\lambda,g) \cdot y = y$. If this set is non empty, it is a coset of $\{y \in Y \mid \al_g(y) = y\}$. If $g\neq e$, the latter is a closed subgroup of $Y$ with infinite index and hence has measure zero. So we may assume that $g = e$ and $\lambda \neq e$, which is the trivial case.
\end{proof}

Take a countable group $\Gamma$ and a compact abelian group $K$ with dense countable subgroup $\Lambda < K$. Put $X := K^\Gamma$ and consider the action $\Gamma \times \Lambda \actson X$ given by $((g,s)\cdot x)_h = s x_{hg}$ for all $g,h \in \Gamma$, $s \in \Lambda$, $x \in X$. Write $M = \rL^\infty(X) \rtimes (\Gamma \times \Lambda)$.

Assume that $K_1 < K$ is a closed subgroup. We view $\rL^\infty(X/K_1)$ as the subalgebra of $K_1$-invariant functions in $\rL^\infty(X)$. We put $\Lambda_1 := \Lambda \cap K_1$ and define $\cC(K_1)$ as the von Neumann subalgebra of $M$ generated by $\rL^\infty(X/K_1)$ and $\rL(\Lambda_1)$. So,
$$\cC(K_1) = \rL^\infty(X/K_1) \rtimes \Lambda_1 \subset M \; .$$
Note that we can identify $\cC(K_1) = \rL^\infty(X/K_1) \ovt \rL \Lambda_1 = \rL^\infty(X/K_1 \times \widehat{\Lambda_1})$.

\begin{lemma}\label{lemma.cartan}
With the above notations, $\cC(K_1)' \cap M = \cC(\overline{\Lambda_1})$. So $\cC(K_1) \subset M$ is maximal abelian iff $\Lambda \cap K_1$ is dense in $K_1$.

In that case, $\cC(K_1)$ is a Cartan subalgebra of $M$ and the induced equivalence relation on $X/K_1 \times \widehat{\Lambda_1}$ is given by the orbits of the product action of $(\Gamma \times \Lambda K_1/K_1) \times \widehat{K_1}$ on $X/K_1 \times \widehat{\Lambda_1}$.
\end{lemma}

For arbitrary closed subgroups $K_1 < K$, the intersection $\Lambda \cap K_1$ need not be dense in $K_1$. But putting $K_2 = \overline{\Lambda \cap K_1}$, we always have that $\Lambda \cap K_2$ is dense in $K_2$.

\begin{proof}
By writing the Fourier decomposition of an element in $M = \rL^\infty(X) \rtimes (\Gamma \times \Lambda)$, one easily checks that
\begin{equation}\label{eq.firstpart}
(\rL \Lambda_1)' \cap M = \rL^\infty(X/\overline{\Lambda_1}) \rtimes (\Gamma \times \Lambda) \; .
\end{equation}
View $X = K^\Gamma$ as a compact group and view $K_1$ as a closed subgroup of $X$ sitting in $X$ diagonally. Put $Y = X/K_1$ and view $\Lambda/\Lambda_1$ as a subgroup of $Y$ sitting in $Y$ diagonally. By Lemma \ref{lemma.free} the action $\Gamma \times \Lambda/\Lambda_1 \actson Y$ is essentially free. From this we conclude that
$$\rL^\infty(X/K_1)' \cap M = \rL^\infty(X) \rtimes \Lambda_1 \; .$$
In combination with \eqref{eq.firstpart} we get that $\cC(K_1)' \cap M = \cC(\overline{\Lambda_1})$.
This proves the first part of the lemma.

Assume now that $\Lambda \cap K_1$ is dense in $K_1$. We know that $\cC(K_1) \subset M$ is maximal abelian. For every $\om \in \widehat{K}$ we define the unitary $U_\om \in \rL^\infty(X)$ given by $U_\om(x) = \om(x_e)$. One checks that $\cC(K_1)$ is normalized by all unitaries in $\cC(K_1)$, $\{U_\om \mid \om \in \widehat{K}\}$, $\{u_g \mid g \in \Gamma\}$ and $\{u_s \mid s \in \Lambda\}$. These unitaries generate $M$ so that $\cC(K_1)$ is indeed a Cartan subalgebra of $M$.

It remains to analyze which automorphisms of $\cC(K_1) = \rL^\infty(X/K_1 \times \widehat{\Lambda_1})$ are induced by the above normalizing unitaries. The automorphism $\Ad U_\om$ only acts on $\widehat{\Lambda_1}$ by first restricting $\om$ to a character on $\Lambda_1$ and then translating by this character on $\widehat{\Lambda_1}$. The automorphism $\Ad u_g$, $g \in \Gamma$, only acts on $X/K_1$ by the quotient of the Bernoulli shift. Finally the automorphism $\Ad u_s$, $s \in \Lambda$, also acts only on $X/K_1$ by first projecting $s$ onto $\Lambda/\Lambda_1 = \Lambda K_1 / K_1$ and then diagonally translating with this element. The resulting orbit equivalence relation indeed corresponds to the direct product of $\cR(\Gamma \times \Lambda K_1/K_1 \actson X/K_1)$ and $\cR(\widehat{K_1} \actson \widehat{\Lambda_1})$.
\end{proof}

\begin{remark}\label{rem.McDuff}
Lemma \ref{lemma.cartan} says that the equivalence relation induced by $\cC(K_1) \subset M$ is the direct product of the orbit relation $\cR((\Gamma \times \Lambda K_1/K_1) \actson X/K_1)$ and the hyperfinite II$_1$ equivalence relation. Nevertheless, if $\Gamma$ is an icc property (T) group, the II$_1$ factor $M$ is not McDuff, i.e.\ cannot be written as a tensor product of some II$_1$ factor with the hyperfinite II$_1$ factor. Indeed, by property (T), all central sequences in $M$ lie asymptotically in the relative commutant $(\rL \Gamma)' \cap M = \rL \Lambda$. Since $\rL \Lambda$ is abelian, it follows that $M$ is not McDuff. So whenever $\Gamma$ is an icc property (T) group, it follows that the $2$-cocycle associated with the Cartan subalgebra $\cC(K_1) \subset M$ is non trivial.
\end{remark}

To prove Theorem \ref{thm.main}, we first have to classify the orbit equivalence relations of the product actions $(\Gamma \times \Lambda K_1/K_1) \times \widehat{K_1}$ on $X/K_1 \times \widehat{\Lambda_1}$, when $K_1$ runs through closed subgroups of $K$ with $\Lambda \cap K_1$ dense in $K_1$. These orbit equivalence relations are the direct product of the orbit equivalence relation of $(\Gamma \times \Lambda K_1 / K_1) \actson X/K_1$ and the unique hyperfinite II$_1$ equivalence relation. In \cite[Theorem 4.1]{PV06}, the actions $(\Gamma \times \Lambda K_1 / K_1) \actson X/K_1$ were classified up to stable orbit equivalence. In Lemma \ref{lemma.reduction}, we redo the proof in the context of a direct product with a hyperfinite equivalence relation. The main difficulty is to make sure that the orbit equivalence ``does not mix up too much'' the first and second factor in the direct product. The orbit equivalence will however not simply split as a direct product of orbit equivalences. The main reason for this is that the II$_1$ factor $N := \rL^\infty(X/K_1) \rtimes (\Gamma \times \Lambda K_1 / K_1)$ has property Gamma. Therefore we cannot apply \cite[Theorem 5.1]{Po06c}, which says the following~: if $N$ is a II$_1$ factor without property Gamma and $R$ is the hyperfinite II$_1$ factor, then every automorphism of $N \ovt R$ essentially splits as a tensor product of two automorphisms.

To state and prove Lemma \ref{lemma.reduction}, we use the following point of view on stable orbit equivalences. Assume that $\cG_i \actson Z_i$, $i=1,2$, are free ergodic p.m.p.\ actions of countable groups $\cG_i$. By definition, a stable orbit equivalence between $\cG_1 \actson Z_1$ and $\cG_2 \actson Z_2$ is a measure space isomorphism $\Delta : \cU_1 \recht \cU_2$ between non negligible subsets $\cU_i \subset Z_i$ satisfying
$$\Delta ( \cG_1 \cdot x \cap \cU_1) = \cG_2 \cdot \Delta(x) \cap \cU_2$$
for a.e.\ $x \in \cU_1$. By ergodicity of $\cG_1 \actson Z_1$, we can choose a measurable map $\Delta_0 : Z_1 \recht \cU_1$ satisfying $\Delta_0(x) \in \cG_1 \cdot x$ for a.e.\ $x \in Z_1$. Denote $\Psi := \Delta \circ \Delta_0$. By construction $\Psi$ is a local isomorphism from $Z_1$ to $Z_2$, meaning that $\Psi : Z_1 \recht Z_2$ is a Borel map such that $Z_1$ can be partitioned into a sequence of non negligible subsets $\cW \subset Z_1$ such that the restriction of $\Psi$ to any of these subsets $\cW$ is a measure space isomorphism of $\cW$ onto some non negligible subset of $Z_2$. Also by construction $\Psi$ is orbit preserving, meaning that for a.e.\ $x,y \in Z_1$ we have that $x \in \cG_1 \cdot y$ iff $\Psi(x) \in \cG_2 \cdot \Psi(y)$.

Using the inverse $\Delta^{-1} : \cU_2 \recht \cU_1$ we analogously find an orbit preserving local isomorphism $\Psib : Z_2 \recht Z_1$. By construction $\Psib(\Psi(x)) \in \cG_1 \cdot x$ for a.e.\ $x \in Z_1$ and $\Psi(\Psib(y)) \in \cG_2 \cdot y$ for a.e.\ $y \in Z_2$. We call $\Psib$ a \emph{generalized inverse} of $\Psi$.

A measurable map $\Psi : Z_1 \recht Z_2$ between standard measure spaces is called a \emph{factor map} iff $\Psi^{-1}(\cV)$ has measure zero whenever $\cV \subset Z_2$ has measure zero. If $\cG_i \actson Z_i$, $i=1,2$, are free ergodic p.m.p.\ actions of the countable groups $\cG_i$ and if $\Psi : Z_1 \recht Z_2$, $\Psib : Z_2 \recht Z_1$ are factor maps satisfying
\begin{align*}
& \Psi(\cG_1 \cdot x) \subset \cG_2 \cdot \Psi(x) \;\;\text{and}\;\; \Psib(\Psi(x)) \in \cG_1 \cdot x \;\;\text{for a.e.}\; x \in Z_1 \;  , \\
& \Psib(\cG_2 \cdot y) \subset \cG_1 \cdot \Psib(y) \;\;\text{and}\;\; \Psi(\Psib(y)) \in \cG_2 \cdot y \;\;\text{for a.e.}\;\; y \in Z_2 \; ,
\end{align*}
then $\Psi$ and $\Psib$ are local isomorphisms, stable orbit equivalences and each other's generalized inverse.

If $\Psi, \Psi' : Z_1 \recht Z_2$ are factor maps that send a.e.\ orbits into orbits, we say that $\Psi$ and $\Psi'$ are \emph{similar} if $\Psi'(x) \in \cG_2 \cdot \Psi(x)$ for a.e.\ $x \in Z_1$.

If $\Gamma \actson (X,\mu)$ and $\Lambda \actson (Y,\eta)$ are p.m.p.\ actions of second countable locally compact groups and if $\delta : \Gamma \recht \Lambda$ is an isomorphism, we call $\Delta : X \recht Y$ a \emph{$\delta$-conjugacy} if $\Delta$ is a measure space isomorphism satisfying $\Delta(g \cdot x) = \delta(g) \cdot \Delta(x)$ a.e.

We start off by the following well known and elementary lemma. For the convenience of the reader, we give a proof.

\begin{lemma}\label{lem.elem}
Let $\cG_i \actson Z_i$ be free ergodic p.m.p.\ actions of countable groups. Let $\Psi : Z_1 \recht Z_2$ be a stable orbit equivalence.
\begin{itemize}
\item If $\Lambda < \cG_1$ is a subgroup and $\Psi(g \cdot x) = \Psi(x)$ for all $g \in \Lambda$ and a.e.\ $x \in Z_1$, then $\Lambda$ is a finite group.
\item If $\delta : \cG_1 \recht \cG_2$ is an isomorphism of $\cG_1$ onto $\cG_2$ and $\Psi(g \cdot x) = \delta(g) \cdot \Psi(x)$ for all $g \in \cG_1$ and a.e.\ $x \in Z_1$, then $\Psi$ is a measure space isomorphism.
\end{itemize}
\end{lemma}
\begin{proof}
To prove the first point take a non negligible subset $\cU \subset Z_1$ such that $\Psi_{|\cU}$ is a measure space isomorphism of $\cU$ onto $\cV \subset Z_2$. Since $\cG_1 \actson Z_1$ is essentially free, it follows that $g \cdot \cU \cap \cU$ is negligible for all $g \in \Lambda - \{e\}$. Hence the sets $(g \cdot \cU)_{g \in \Lambda}$ are essentially disjoint. Since $Z_1$ has finite measure and $\cU$ is non negligible, it follows that $\Lambda$ is finite.

We now prove the second point. Since $\Psi$ is a local isomorphism, the image $\Psi(Z_1)$ is a measurable, $\delta(\cG_1)$-invariant, non negligible subset of $Z_2$. Since $\delta(\cG_1) = \cG_2$ and $\cG_2 \actson Z_2$ is ergodic, it follows that $\Psi$ is essentially surjective. Assume that $\Psi$ is not a measure space isomorphism. Since $\Psi$ is a local isomorphism, we find non negligible disjoint subsets $\cU_1,\cU_2 \subset Z_1$ such that $\Psi_{|\cU_i}$, $i =1,2$, are measure space isomorphisms onto the same subset $\cV \subset Z_2$. Since $\Psi$ is a stable orbit equivalence, $(\Psi_{|\cU_2})^{-1}(\Psi(x)) \in \cG_1 \cdot x$ for a.e.\ $x \in \cU_1$. Making $\cU_1, \cU_2$ smaller, we may assume that there exists a $g \in \cG_1$ such that $(\Psi_{|\cU_2})^{-1}(\Psi(x)) = g \cdot x$ for a.e.\ $x \in \cU_1$. Since $\cU_1$ and $\cU_2$ are disjoint, it follows that $g \neq e$. But it also follows that $\delta(g) \cdot \Psi(x) = \Psi(x)$ for a.e.\ $x \in \cU_1$. Since $\cG_2 \actson Z_2$ is essentially free, we arrive at the contradiction that $\delta(g) = e$.
\end{proof}

\begin{lemma}\label{lemma.reduction}
Let $\Gamma$ be an icc property (T) group with $[\Gamma,\Gamma] = \Gamma$. As above take a compact abelian group $K$ with dense countable subgroup $\Lambda < K$. Put $X := K^\Gamma$.

Let $K_i < K$, $i=1,2$, be closed subgroups. Denote $G_i := \Gamma \times \Lambda K_i$ with its natural action on $X$ given by
$$G_i \actson K^\Gamma : ((g,k) \cdot x)_h = k \, x_{hg} \quad\text{for all}\;\; g,h \in \Gamma, k \in \Lambda K_i, x \in K^\Gamma \; .$$
Let $H_i \actson Y_i$ be free ergodic p.m.p.\ actions of the countable abelian groups $H_i$. Assume that
$$\Delta : X/K_1 \times Y_1 \recht X/K_2 \times Y_2$$
is a stable orbit equivalence between the product actions $G_i/K_i \times H_i \actson X/K_i \times Y_i$. Then there exists
\begin{itemize}
\item a compact subgroup $K'_1 < \Lambda K_1$ that is commensurate with $K_1$, with corresponding canonical stable orbit equivalence $\Delta_1 : X/K'_1 \recht X/K_1$ between the actions $G_1/K'_1 \actson X/K'_1$ and $G_1/K_1 \actson X/K_1$;
\item an automorphism $\delta_1 \in \Aut \Gamma$ and an isomorphism $\delta_2 : \Lambda K_1/K'_1 \recht \Lambda K_2/K_2$;
\item a stable orbit equivalence $\Psi : Y_1 \recht Y_2$ between the actions $H_i \actson Y_i$;
\item a measurable family $(\Theta_y)_{y \in Y_1}$ of $(\delta_1 \times \delta_2)$-conjugacies between the actions $G_1/K'_1 \actson X/K'_1$ and $G_2/K_2 \actson X/K_2$;
\item a $\delta_1$-conjugacy $\Theta_1$ of $\Gamma \actson X/K$;
\end{itemize}
such that the stable orbit equivalence $\Delta \circ (\Delta_1 \times \id)$ is similar with the stable orbit equivalence
$$\Theta : X/K'_1 \times Y_1 \recht X/K_2 \times Y_2 : (x,y) \mapsto (\Theta_y(x),\Psi(y))$$
and such that $\Theta_y(x K_1')K = \Theta_1(xK)$ a.e.
\end{lemma}

Note that we turn $G_i$ into a locally compact group by requiring that $K_i < G_i$ is a compact open subgroup.

\begin{proof}
Denote
$$\Theta : X \times Y_1 \recht X/K_2 \times Y_2 : \Theta(x,y) = \Delta(x K_1,y) \; .$$
Associated with the orbit map $\Theta$ are two measurable families of cocycles:
\begin{align*}
\om^y : G_1 \times X \recht G_2/K_2 \times H_2 &\quad\text{determined by}\;\; \Theta(g \cdot x,y) = \om^y(g,x) \cdot \Theta(x,y) \; ,\\
\mu^x : H_1 \times Y_1 \recht G_2/K_2 \times H_2 &\quad\text{determined by}\;\; \Theta(x, h \cdot y) = \mu^x(h,y) \cdot \Theta(x,y) \; .
\end{align*}
By a combination of Popa's cocycle superrigidity theorem for $\Gamma \actson K^\Gamma$ (see \cite[Theorem 0.1]{Po05}) and the elementary \cite[Lemma 5.5]{PV08}, we can replace $\Theta$ by a similar factor map and find measurable families of continuous group homomorphisms
$$\delta_y : G_1 \recht G_2/K_2 \quad\text{and}\quad \eta_y : G_1 \recht H_2$$
such that
$$\om^y(g,x) = (\delta_y(g),\eta_y(g)) \quad\text{and}\quad \Theta(g \cdot x,y) = (\delta_y(g),\eta_y(g)) \cdot \Theta(x,y) \quad\text{a.e.}$$
Since $[\Gamma,\Gamma]=\Gamma$ and $H_2$ is abelian, it follows that $\eta_y(g) = e$ for all $g \in \Gamma$ and a.e.\ $y \in Y_1$. Writing $\Theta(x,y) = (\Theta_1(x,y),\Theta_2(x,y))$, it follows that $\Theta_2(g \cdot x,y) = \Theta_2(x,y)$ for all $g \in \Gamma$ and a.e.\ $(x,y) \in X \times Y_1$. By ergodicity of $\Gamma \actson X$, it follows that $\Theta_2(x,y) = \Psi(y)$ for some factor map $\Psi : Y_1 \recht Y_2$. But then
$$\eta_y(g) \cdot \Psi(y) = \eta_y(g) \cdot \Theta_2(x,y) = \Theta_2(g \cdot x, y) = \Psi(y)$$
for all $g \in G_1$ and a.e.\ $(x,y) \in X \times Y_1$. It follows that $\eta_y(g) = e$ for all $g \in G_1$ and a.e.\ $y \in Y_1$.

Denote $\mu^x(h,y) = (\mu^x_1(h,y),\mu^x_2(h,y))$. Because
$$\mu^x_2(h,y) \cdot \Psi(y) = \mu^x_2(h,y) \cdot \Theta_2(x,y) = \Theta_2(x,h \cdot y) = \Psi(h \cdot y)$$
for all $h \in H_1$ and a.e.\ $(x,y) \in X \times Y_1$, it follows that $\mu^x_2(h,y)$ is essentially independent of the $x$-variable. So we write $\mu_2(h,y)$ instead of $\mu^x_2(h,y)$.

Next we observe that $\Theta_1(g \cdot x,h \cdot y)$ can be computed in two ways, leading to the equality
$$\delta_{h \cdot y}(g) \, \mu^x_1(h,y) \cdot \Theta_1(x,y)  =  \mu^{g \cdot x}_1 (h,y) \, \delta_y(g) \cdot \Theta_1(x,y) \quad\text{a.e.}$$
So, $\mu^{g \cdot x}_1(h,y) = \delta_{h \cdot y}(g) \, \mu^x_1(h,y) \, \delta_y(g)^{-1}$ for all $g \in G_1$, $h \in H_1$ and a.e.\ $(x,y) \in X \times Y_1$.

It follows that for all $h \in H_1$ and a.e.\ $y \in Y_1$, the non negligible set $\{(x,x') \in X \times X \mid \mu^x_1(h,y) = \mu^{x'}_1(h,y)\}$ is invariant under the diagonal action of $\Gamma$ on $X \times X$. Since $\Gamma \actson X$ is weakly mixing, we conclude that $\mu^x_1(h,y)$ is essentially independent of the $x$-variable. We write $\mu_1(h,y)$ instead of $\mu^x_1(h,y)$. The above formula becomes
\begin{equation}\label{eq.delta}
\delta_{h \cdot y}(g) = \mu_1(h,y) \, \delta_y(g) \, \mu_1(h,y)^{-1} \; .
\end{equation}
It follows that the map $y \mapsto \Ker \delta_y$ is $H_1$-invariant and hence, by ergodicity of $H_1 \actson Y_1$, there is a unique closed subgroup $K'_1 < G_1$ such that $\Ker \delta_y = K'_1$ for a.e.\ $y \in Y_1$. Since $G_2/K_2$ is a countable group, we know that its compact subgroup $\delta_y(K_1)$ must be finite. Hence, $K'_1 \cap K_1 < K_1$ has finite index. We claim that also $K'_1 \cap K_1 < K'_1$ has finite index.

To prove this claim, first observe that the formula $\Theta(g \cdot x,y) = \delta_y(g) \cdot \Theta(x,y)$ implies that we can view $\Theta$ as a map from $X/K'_1 \times Y_1$ to $X/K_2 \times Y_2$. In particular $\Theta$ can be viewed as a map from $X/(K_1 \cap K'_1) \times Y_1$ to $X/K_2 \times Y_2$. Since $X/(K_1 \cap K'_1)$ is a finite covering of $X/K_1$ and since the original $\Delta$ was a local isomorphism, also
$$\Theta : X/(K_1 \cap K'_1) \times Y_1 \recht X/K_2 \times Y_2$$
is a local isomorphism and hence a stable orbit equivalence. Since moreover $\Theta(g \cdot x,y) = \Theta(x,y)$ for all $g \in K'_1/(K_1 \cap K'_1)$, it follows from Lemma \ref{lem.elem} that $K'_1/(K_1 \cap K'_1)$ is a finite group. This proves the claim, meaning that $K'_1$ and $K_1$ are commensurate subgroups of $G_1$. In particular, $K_1'$ is compact. Denoting by $p_\Gamma : G_1 \recht \Gamma$ the projection homomorphism onto the first factor, we get that $p_\Gamma(K_1')$ is a finite normal subgroup of $\Gamma$. Since $\Gamma$ is icc, it follows that $p_\Gamma(K_1') = \{e\}$ and thus $K'_1 < \Lambda K_1$.

By construction the map
$$\Theta : X/K'_1 \times Y_1 \recht X/K_2 \times Y_2$$
is a stable orbit equivalence with the following properties:
\begin{itemize}
\item $\Theta$ is similar with $\Delta \circ (\Delta_1 \times \id)$~;
\item $\Theta$ is of the form $\Theta(x,y) = (\Theta_y(x), \Psi(y))$~;
\item and $\Theta$ satisfies $\Theta_y(g \cdot x) = \delta_y(g) \cdot \Theta_y(x)$ a.e.\ where $\delta_y : G_1/K'_1 \recht G_2/K_2$ is a measurable family of injective group homomorphisms.
\end{itemize}
It remains to prove that the $\delta_y$ are group isomorphisms that do not depend on $y$, that $\Psi$ is a stable orbit equivalence and that a.e.\ $\Theta_y$ is a measure space isomorphism of $X/K'_1$ onto $X/K_2$.

Applying the same reasoning to the inverse of the stable orbit equivalence $\Delta \circ (\Delta_1 \times \id)$, we find a closed subgroup $K'_2 < \Lambda K_2$ that is commensurate with $K_2$, a measurable family of injective group homomorphisms $\deltab_y : G_2/K'_2 \recht G_1/K'_1$ and a map
$$\Thetab : X/K'_2 \times Y_2 \recht X/K'_1 \times Y_1$$
of the form $\Thetab(x,y) = (\Thetab_y(x),\Psib(y))$ satisfying the following two properties:
\begin{align}
&\Thetab_y(g \cdot x) = \deltab_y(g) \cdot \Thetab_y(x) \quad\text{and} \notag\\
&\Theta(\Thetab(x K'_2,y)) \in (G_2/K_2 \times H_2) \cdot (x K_2,y) \quad\text{a.e.}\label{eq.orbits}
\end{align}
The second of these formulae yields a measurable map $\vphi : X \times Y_2 \recht G_2/K_2$ such that
$$\Theta_{\Psib(y)}(\Thetab_y(x K'_2)) = \vphi(x,y) \cdot xK_2 \quad\text{a.e.}$$
Computing $\Theta_{\Psib(y)}(\Thetab_y(gK'_2 \cdot x K'_2))$ it follows that
$$\delta_{\Psib(y)}(\deltab_y(gK'_2)) \, \vphi(x,y) = \vphi(g \cdot x,y) \, gK_2 \quad\text{for all $g \in G_2$ and a.e.\ $(x,y) \in X \times Y_2$.}$$
So, for almost every $y \in Y_2$, the non negligible set $\{(x,x') \in X \times X \mid \vphi(x,y) = \vphi(x',y)\}$ is invariant under the diagonal action of $\Gamma$ on $X \times X$.
Since $\Gamma \actson X$ is weakly mixing, it follows that $\vphi(x,y)$ is essentially independent of the $x$-variable. We write $\vphi(y)$ instead of $\vphi(x,y)$ and find that
$$(\delta_{\Psib(y)} \circ \deltab_y)(g K'_2) = \vphi(y) \, g K_2 \, \vphi(y)^{-1} \; .$$
It follows that $K'_2 = K_2$ and that $\delta_{\Psib(y)} \circ \deltab_y$ is an inner automorphism of $G_2/K_2$. It follows in particular that $\delta_{\Psib(y)}$ is surjective, so that $\delta_y$ is an isomorphism of $G_1/K'_1$ onto $G_2/K_2$ for all $y$ in a non negligible subset of $Y_1$. Because of \eqref{eq.delta} and the ergodicity of $H_1 \actson Y_1$, it follows that $\delta_y$ is an isomorphism for a.e.\ $y \in Y_1$.

It also follows from \eqref{eq.orbits} that $\Psi(\Psib(y)) \in H_2 \cdot y$ for a.e.\ $y \in Y_2$ and similarly with $\Psib \circ \Psi$. Hence $\Psi$ and $\Psib$ are local isomorphisms and hence a stable orbit equivalence between $H_1 \actson Y_1$ and $H_2 \actson Y_2$. Since $\Theta(x,y) = (\Theta_y(x), \Psi(y))$ and since $\Theta$ is a local isomorphism, it then follows that a.e.\ $\Theta_y$ is a local isomorphism. But $\Theta_y$ is also a $\delta_y$-conjugation w.r.t.\ the isomorphism $\delta_y$. Lemma \ref{lem.elem} implies that $\Theta_y$ is a measure space isomorphism for a.e.\ $y \in Y_1$.

We know that $[\Gamma,\Gamma] = \Gamma$, that $\Gamma$ has trivial center and that the groups $\Lambda K_i/K'_i$ are abelian. So it follows that $\delta_y$ is the direct product of an automorphism $\delta^1_y$ of $\Gamma$ and an isomorphism $\delta^2_y$ of $\Lambda K_1/K'_1$ onto $\Lambda K_2/K_2$. Formula \eqref{eq.delta} and the fact that $\Lambda K_2/K_2$ is the center of $G_2/K_2$ imply that $\delta^2_{h \cdot y} = \delta^2_y$ a.e. By ergodicity of $H_1 \actson Y_1$, it follows that $\delta^2_y$ is essentially independent of $y$. We write $\delta_2$ instead of $\delta^2_y$. Next denote by $\mu_\Gamma(h,y)$ the projection of $\mu_1(h,y) \in \Gamma \times \Lambda K_2/K_2$ onto $\Gamma$. Formula \eqref{eq.delta} says that
$$\delta^1_{h \cdot y} = \Ad \mu_\Gamma(h,y) \circ \delta^1_y \; .$$
Since $\Gamma$ has property (T), $\Gamma$ is finitely generated and $\Aut \Gamma$ is a countable group. In particular, $\Out \Gamma = \Aut \Gamma / \Inn \Gamma$ is a countable group. Then the map $y \mapsto (\delta^1_y \;\text{mod}\; \Inn \Gamma)$ is an $H_1$-invariant measurable map from $Y_1$ to $\Out \Gamma$ and hence is essentially constant. So we find a measurable map $\vphi : Y_1 \recht \Gamma$ and an automorphism $\delta_1 \in \Aut \Gamma$ such that $\delta^1_y = \Ad \vphi(y) \circ \delta_1$ a.e. Replacing $\Theta_y(x)$ by $\vphi(y)^{-1} \cdot \Theta_y(x)$ we may assume that $\delta_y$ is a.e.\ equal to $\delta_1 \times \delta_2$. After this replacement, also $\mu_\Gamma(h,y) = e$ a.e.\ meaning that $\mu_1(h,y) \in \Lambda K_2/K_2$.

Since $\Theta_y$ is a $\delta_2$-conjugacy and $\Lambda K_1, \Lambda K_2$ are dense subgroups of $K$, the $\Theta_y$ induce $\delta_1$-conjugacies $\rho_y$ of the action $\Gamma \actson X/K$ given by $\rho_y(xK) = \Theta_y(x K'_1) K$. Since $\mu_1(h,y) \in \Lambda K_2/K_2$, it follows that $\rho_{h \cdot y}(x) = \rho_y(x)$ a.e. Hence $\rho_y$ is a.e.\ equal to a single $\delta_1$-conjugacy $\Theta_1$ of the action $\Gamma \actson X/K$.
\end{proof}

\begin{lemma}\label{lem.density}
Let $K$ be a compact abelian group with countable dense subgroup $\Lambda$. Let $K_1, K_2 < K$ be closed subgroups and denote $\Lambda_i := \Lambda \cap K_i$.
\begin{enumerate}
\item If $K_2 \subset \Lambda K_1$ and if $\Lambda K_1 / K_2$ is countable, then $K_1$ and $K_2$ are commensurate.

\item Assume that $K_1$ and $K_2$ are commensurate and that $K_2 \subset \Lambda K_1$. If $\Lambda_1$ is dense in $K_1$, also $\Lambda_2$ is dense in $K_2$.

\item Assume that $\Lambda_i$ is dense in $K_i$ for $i=1,2$. Then the following statements are equivalent.
\begin{enumerate}
\item $\Lambda_1$ and $\Lambda_2$ are commensurate.
\item $K_1$ and $K_2$ are commensurate.
\item $\Lambda K_1 = \Lambda K_2$.
\end{enumerate}
\end{enumerate}
\end{lemma}
\begin{proof}
1.\ Since $\Lambda K_1 / K_2$ is countable, the image of $K_1$ in $K/K_2$ is a countable compact group, hence a finite group. This means that $K_1 \cap K_2 < K_1$ has finite index. Since $K_2 \subset \Lambda K_1$, also the image of $K_2$ in $K/K_1$ is countable, hence finite. This means that $K_1 \cap K_2 < K_2$ has finite index. So, $K_1$ and $K_2$ are commensurate.

2.\ Assume that $K_1$ and $K_2$ are commensurate, that $K_2 \subset \Lambda K_1$ and that $\Lambda_1$ is dense in $K_1$.
Since $K_1 \cap K_2 < K_1$ has finite index, $K_1 \cap K_2$ is an open subgroup of $K_1$. Because $\Lambda_1$ is dense in $K_1$, we get that $\Lambda_1 \cap (K_1 \cap K_2)$ is dense in $K_1 \cap K_2$. So, $\Lambda_1 \cap \Lambda_2$ is dense in $K_1 \cap K_2$. Also the image of $\Lambda_1$ in $K_1 / (K_1 \cap K_2)$ is dense. But $K_1 / (K_1 \cap K_2)$ is finite and it follows that $K_1 = \Lambda_1 (K_1 \cap K_2)$. Hence also
$\Lambda K_1 = \Lambda (K_1 \cap K_2)$. Since by assumption $K_2 \subset \Lambda K_1$, we get that $K_2 \subset \Lambda (K_1 \cap K_2)$, i.e.\ $K_2 = \Lambda_2 (K_1 \cap K_2)$. We already saw that $\Lambda_1 \cap \Lambda_2$ is dense in $K_1 \cap K_2$ and it follows that $\Lambda_2$ is dense in $K_2$.

3.\ (a) $\Rightarrow$ (b).\ The image of $\Lambda_i$ in $K_i/(K_1 \cap K_2)$ is dense, but also finite because $\Lambda_1 \cap \Lambda_2$ has finite index in $\Lambda_i$. So, $K_i/(K_1 \cap K_2)$ is finite for $i=1,2$, meaning that $K_1$ and $K_2$ are commensurate.

(b) $\Rightarrow$ (c).\ Since $\Lambda_i$ is dense in $K_i$ and $K_1 \cap K_2 < K_i$ has finite index, we get as in the proof of 2 that $K_i = \Lambda_i(K_1 \cap K_2)$. Hence, $\Lambda K_i = \Lambda (K_1 \cap K_2)$ for $i=1,2$. In particular, $\Lambda K_1 = \Lambda K_2$.

(c) $\Rightarrow$ (b).\ We have $K_2 \subset \Lambda K_1$ and we have that $\Lambda K_1 / K_2 = \Lambda K_2 / K_2$ is countable. So, statement 1 implies that $K_1$ and $K_2$ are commensurate.

(b) $\Rightarrow$ (a).\ Since $K_i / (K_1 \cap K_2)$ is finite, also the image of $\Lambda_i$ in $K_i / (K_1 \cap K_2)$ is finite, meaning that $\Lambda_i / (\Lambda_1 \cap \Lambda_2)$ is a finite group. So, $\Lambda_1$ and $\Lambda_2$ are commensurate.
\end{proof}

\subsubsection*{Proof of Theorem \ref{thm.main}}

\begin{proof}
{\bf Proof of the equivalence of the statements in 1.} Put $\Lambda_i = \Lambda \cap K_i$. Using Lemma \ref{lem.density}, it suffices to prove that $\cC(K_1)$ and $\cC(K_2)$ are unitarily conjugate if and only if $\Lambda_1$ and $\Lambda_2$ are commensurate. First assume that $\Lambda_1$ and $\Lambda_2$ are commensurate. By Lemma \ref{lem.density}, also $K_1$ and $K_2$ are commensurate. So, $K_i < K_1 K_2$ has finite index. It follows that $\rL^\infty(K^\Gamma / K_1 K_2) \ovt \rL (\Lambda_1 \cap \Lambda_2)$ is a subalgebra of finite index in both $\cC(K_1)$ and $\cC(K_2)$. By \cite[Theorem A.1]{Po01} the Cartan subalgebras $\cC(K_1)$ and $\cC(K_2)$ are unitarily conjugate.

Conversely, assume that $\Lambda_1$ and $\Lambda_2$ are not commensurate. Assume for instance that $\Lambda_1 \cap \Lambda_2 < \Lambda_1$ has infinite index. We find a sequence $\lambda_n \in \Lambda_1$ such that for every finite subset $\cF \subset \Lambda$ the element $\lambda_n$ lies outside $\cF \Lambda_2$ eventually. Denote by $u_n := u_{\lambda_n}$ the unitaries in $\cC(K_1)$ that correspond to the group elements $\lambda_n$. Denoting by $E_{\cC(K_2)}$ the trace preserving conditional expectation, one deduces that
$$\|E_{\cC(K_2)}(a u_n b)\|_2 \recht 0 \quad\text{for all}\quad a,b \in M \; .$$
So, by \cite[Theorem A.1]{Po01}, the Cartan subalgebras $\cC(K_1)$ and $\cC(K_2)$ are not unitarily conjugate.

{\bf Proof of the equivalence of the statements in 3.} Denote $X = K^\Gamma$. Assume first that there exists a continuous automorphism $\delta \in \Aut K$ satisfying $\delta(\Lambda K_1) = \Lambda K_2$. Then $K_1' := \delta^{-1}(K_2)$ is a closed subgroup of $K$ with the properties that $K_1' \subset \Lambda K_1$ and that $\Lambda K_1 / K_1'$ is countable. By Lemma \ref{lem.density}.1, we get that $K_1'$ and $K_1$ are commensurate. Then Lemma \ref{lem.density}.2 implies that $\Lambda \cap K_1'$ is dense in $K_1'$. By the equivalence of the statements in 1, we know that the Cartan subalgebras $\cC(K_1')$ and $\cC(K_1)$ are unitarily conjugate. In particular, they give rise to isomorphic equivalence relations. So we may replace $K_1$ by $K_1'$ and assume that $\delta(K_1) = K_2$. By ``applying everywhere'' $\delta$, the actions $(\Gamma \times \Lambda K_i/K_i) \actson X/K_i$ are conjugate, hence orbit equivalent. Also their respective direct products with the hyperfinite equivalence relation are orbit equivalent. By Lemma \ref{lemma.cartan}, the equivalence relations induced by $\cC(K_i)$, $i = 1,2$, are orbit equivalent.

Conversely assume that the equivalence relations induced by $\cC(K_i)$, $i=1,2$, are stably orbit equivalent. Put $G_i := \Gamma \times \Lambda K_i$. We combine Lemmas \ref{lemma.cartan} and \ref{lemma.reduction}, and we replace $K_1$ by a commensurate group that is contained in $\Lambda K_1$. By Lemma \ref{lem.density}.2, the intersection $\Lambda \cap K_1$ remains dense in $K_1$. We have found a measure space isomorphism $\Delta : X/K_1 \recht X/K_2$, an automorphism $\delta_1 \in \Aut \Gamma$ and an isomorphism $\delta_2 : \Lambda K_1 / K_1 \recht \Lambda K_2 / K_2$ such that $\Delta$ is a $(\delta_1 \times \delta_2)$-conjugacy. In particular, $\Delta$ is a $\delta_1$-conjugacy of the $\Gamma$-action. By \cite[Lemma 5.2]{PV06} there exists an isomorphism $\al : K_1 \recht K_2$ and a $(\delta_1 \times \al)$-conjugacy $\Deltab : X \recht X$ satisfying $\Deltab(x) K_2 = \Delta(x K_1)$ for a.e.\ $x \in X$. Since $\Delta$ is a $\delta_2$-conjugacy, it follows that $\Deltab(g \cdot x) \in \Lambda K_2 \cdot \Delta(x)$ for all $g \in \Lambda K_1$ and a.e.\ $x \in X$. We find a measurable map $\om : \Lambda K_1 \times X \recht \Lambda K_2$ such that $\Deltab(g \cdot x) = \om(g,x) \cdot \Deltab(x)$ a.e. Since $\Deltab$ is a $\delta_1$-conjugacy, it follows that $\om(g,\gamma \cdot x) = \om(g,x)$ for all $\gamma \in \Gamma$, $g \in \Lambda K_1$ and a.e.\ $x \in X$. By ergodicity of $\Gamma \actson X$, the map $\om(g,x)$ is essentially independent of the $x$-variable. So we can extend $\al$ to a homomorphism $\al : \Lambda K_1 \recht \Lambda K_2$. By symmetry (i.e.\ considering ${\Deltab}^{-1}$) it follows that $\al$ is an isomorphism of $\Lambda K_1$ onto $\Lambda K_2$. Viewing $K$ as a closed subgroup of the group $\Aut(X)$ of measure preserving automorphisms of $X$ (up to equality almost everywhere), it follows that $\Deltab \Lambda K_1 {\Deltab}^{-1} = \Lambda K_2$. Taking closures we also have $\Deltab K {\Deltab}^{-1} = K$. This means that $\al$ can be extended to a continuous automorphism of $K$. This concludes the proof of the equivalence of the statements in 3.

{\bf Proof of the equivalence of the statements in 2.} Denote $X = K^\Gamma$. If $\delta \in \Aut K$ is a continuous automorphism satisfying $\delta(\Lambda) = \Lambda$ and $\delta(\Lambda K_1) = \Lambda K_2$, we may replace as above $K_1$ by the commensurate $\delta^{-1}(K_2)$ and assume that $\delta(K_1) = K_2$. Then also $\delta(\Lambda_1) = \Lambda_2$. So, ``applying everywhere'' $\delta$ yields an automorphism $\al \in \Aut M$ satisfying $\al(\cC(K_1)) = \cC(K_2)$.

Conversely assume that $\cC(K_1)$ and $\cC(K_2)$ are stably conjugate by an automorphism. Interchanging the roles of $K_1$ and $K_2$ if necessary, we find a projection $p \in \cC(K_2)$ and an isomorphism $\al : M \recht pMp$ satisfying $\al(\cC(K_1)) = \cC(K_2) p$. This $\al$ induces a stable orbit equivalence between the equivalence relations associated with $\cC(K_1)$ and $\cC(K_2)$ respectively. If one of the $K_i$ is finite, the equivalence of the statements in 3 implies that the other one is finite as well and that $\Lambda = \Lambda K_1 = \Lambda K_2$. So we can exclude this trivial case.

We claim that there exist automorphisms $\delta_1 \in \Aut \Gamma$, $\delta \in \Aut K$ and a $(\delta_1 \times \delta)$-conjugacy $\Delta \in \Aut(X)$ of the action $\Gamma \times K \actson X$ such that $\delta(K_1) = K_2$ and such that after a unitary conjugacy of $\al$ and $p$, we have
$$p \in \rL \Lambda_2 \quad\text{and}\quad \al(F) = (F \circ \Delta^{-1}) \, p \quad\text{for all}\;\; F \in \rL^\infty(X/K) \; .$$
To prove this claim, denote $Y_i = \widehat{\Lambda_i}$ and $H_i = \widehat{K_i}$. Also denote $G_i = \Gamma \times \Lambda K_i$. Lemma \ref{lemma.cartan} describes the equivalence relations associated with $\cC(K_i)$. The isomorphism $\al : M \recht pMp$ therefore induces a stable orbit equivalence $\Pi : X/K_1 \times Y_1 \recht X/K_2 \times Y_2$ between the product actions $G_i /K_i \times H_i \actson X/K_i \times Y_i$. We apply Lemma \ref{lemma.reduction} to $\Pi$. We find in particular a compact subgroup $K'_1 < \Lambda K_1$ that is commensurate with $K_1$. Denote $\Lambda'_1 := \Lambda \cap K'_1$, $H'_1 := \widehat{K'_1}$ and $Y'_1 := \widehat{\Lambda'_1}$. By Lemma \ref{lem.density}.2, $\Lambda_1'$ is dense in $K_1'$. By the equivalence of the statements in 1, we know that the Cartan subalgebras $\cC(K'_1)$ and $\cC(K_1)$ are unitarily conjugate. The corresponding orbit equivalence between the product actions $G_1/K'_1 \times H'_1 \actson X/K'_1 \times Y'_1$ and $G_1/K_1 \times H_1 \actson X/K_1 \times Y_1$ is similar with $\Delta_1 \times \Delta_2$ where $\Delta_1,\Delta_2$ are the natural stable orbit equivalences (whose compression constants are each other's inverse). So replacing $\cC(K_1)$ by $\cC(K'_1)$ amounts to replacing $\Pi$ by $\Pi \circ (\Delta_1 \times \Delta_2)$. Lemma \ref{lemma.reduction} says that $\Pi \circ (\Delta_1 \times \id)$ is similar to a stable orbit equivalence of a very special form. Then $\Pi \circ (\Delta_1 \times \Delta_2)$ is still of this very special form. So after replacing $\cC(K_1)$ by $\cC(K'_1)$ we find that $\Pi$ is similar to a stable orbit equivalence $\Theta : X/K_1 \times Y_1 \recht X/K_2 \times Y_2$ satisfying the following properties.
\begin{itemize}
\item $\Theta$ is of the form $\Theta(x,y) = (\Theta_y(x),\Psi(y))$ where $\Psi : Y_1 \recht Y_2$ is a stable orbit equivalence between the actions $H_i \actson Y_i$ and where $(\Theta_y)_{y \in Y_1}$ is a measurable family of measure preserving conjugacies between $G_1/K_1 \actson X/K_1$ and $G_2/K_2 \actson X/K_2$.
\item There exist automorphisms $\delta_1 \in \Aut \Gamma$, $\delta \in \Aut K$ and a $(\delta_1 \times \delta)$-conjugacy $\Delta \in \Aut(X)$ such that $\delta(K_1) = K_2$ and $\Theta_y(x)K = \Delta(xK)$ for a.e.\ $(x,y) \in X/K_1 \times Y_1$.
\end{itemize}
To deduce the precise form of $\Delta$ in the last item, we proceed as in the proof above of the equivalence of the statements in 3. The compression constant of $\Psi$ is equal to the compression constant of $\Pi$ and hence equal to $\tau(p)$. So replacing $\Psi$ by a similar stable orbit equivalence, we may assume that $\Psi$ is a measure space isomorphism (scaling the measure by the factor $\tau(p)$) of $Y_1$ onto a measurable subset $\cU \subset Y_2$ of measure $\tau(p)$. Keeping $\Theta_y(x)$ unchanged, the map $\Theta$ then becomes a measure space isomorphism of $X/K_1 \times Y_1$ onto $X/K_2 \times \cU$. Define the projection $p \in \rL \Lambda_2 = \rL^\infty(Y_2)$ by $p = \chi_\cU$. Since $\Pi$ and $\Theta$ are similar stable orbit equivalences and $\Theta$ is moreover a measure space isomorphism, we can unitarily conjugate $\al$ so that $\al_{|\cC(K_1)} = \Theta_*$. This proves the claim above.

It remains to prove that $\delta(\Lambda) = \Lambda$.

Since every automorphism $\delta_1 \in \Aut \Gamma$ defines a natural automorphism of $M$ that globally preserves all the Cartan subalgebras $\cC(K_1)$, we may assume that $\delta_1 = \id$. For $g \in \Gamma \times \Lambda$, we denote by $u_g \in M = \rL^\infty(X) \rtimes (\Gamma \times \Lambda)$ the canonical unitaries. Since the relative commutant of $\rL^\infty(X/K)$ inside $M$ equals $\rL^\infty(X) \rtimes \Lambda$, it follows that
$$\al(u_g) = \om_g u_g p \quad\text{for all $g \in \Gamma$, where}\quad \om_g \in \cU(p(\rL^\infty(X) \rtimes \Lambda)p) \; .$$
Denote by $(\si_g)_{g \in \Gamma}$ the action of $\Gamma$ on $\rL^\infty(X) \rtimes \Lambda$, implemented by $\Ad u_g$ and corresponding to the Bernoulli action on $\rL^\infty(X)$ and the trivial action on $\rL \Lambda$. Since $p$ is invariant under $(\si_g)_{g \in \Gamma}$ we can restrict $\si_g$ to $p(\rL^\infty(X) \rtimes \Lambda)p$. Note that
$$\om_{gh} = \om_g \, \si_g(\om_h) \quad\text{for all}\;\; g,h \in \Gamma \; .$$
Denote $N^\infty := \B(\ell^2(\N)) \ovt N$ whenever $N$ is a von Neumann algebra.
By Theorem \ref{thm.cocycle} there exists a projection $q \in (\rL \Lambda)^\infty$ with $(\Tr \ot \tau)(q) = \tau(p)$, a partial isometry $v \in \B(\C,\ell^2(\N)) \ovt (\rL^\infty(X) \rtimes \Lambda)$ and a group homomorphism $\gamma : \Gamma \recht \cU(q (\rL \Lambda)^\infty q)$ satisfying
$$v^* v = p \;\; , \;\; vv^* = q \quad\text{and}\quad \om_g = v^* \, \gamma_g \, \si_g(v) \;\;\text{for all}\;\; g \in \Gamma \; .$$
We replace $\al$ by $\Ad v \circ \al$. Then $\al : M \recht q M^\infty q$. Note that $\rL^\infty(X/K)$ commutes with $v$ so that
$$\al(F) = (F \circ \Delta^{-1}) q \;\;\text{for all}\;\; F \in \rL^\infty(X/K) \quad\text{and}\quad \al(u_g) = \gamma_g u_g = u_g \gamma_g \;\;\text{for all}\;\; g \in \Gamma \; .$$
Denote $\Delta_* \in \Aut(\rL^\infty(X))$ given by $\Delta_*(F) = F \circ \Delta^{-1}$ and note that $\Delta_*(\rL^\infty(X/K)) = \rL^\infty(X/K)$. Also $\al(F) = \Delta_*(F) q$ whenever $F \in \rL^\infty(X/K)$.

Recall that $X = K^\Gamma$. Whenever $\om \in \widehat{K}$, denote by $U_\om \in \cU(\rL^\infty(X))$ the unitary given by $U_\om(x) = \om(x_e)$. Fix $\om \in \widehat{K}$ and $g \in \Gamma$. Observe that $U_\om \, u_g U_\om^* u_g^*$ belongs to $\rL^\infty(X/K)$. Hence, since $\Delta$ is an $\id$-conjugacy for $\Gamma \actson X$, we have
\begin{align*}
\al(U_\om \, u_g U_\om^* u_g^*) &= \Delta_*(U_\om \, \si_g(U_\om^*)) q = \Delta_*(U_\om) \,  \si_g(\Delta_*(U_\om^*)) \, q \\ &= \Delta_*(U_\om) \, u_g \, \Delta_*(U_\om)^* \, u_g^* \, q \; .
\end{align*}
On the other hand
$$\al(U_\om \, u_g U_\om^* u_g^*) = \al(U_\om) \, u_g \, \gamma_g \, \al(U_\om)^* \, \gamma_g^* u_g^* \; .$$
Combining both formulae we conclude that
\begin{equation}\label{eq.useful}
u_g^* \; \Delta_*(U_\om)^* \, \al(U_\om) \; u_g = \Delta_*(U_\om)^* \, \gamma_g \, \al(U_\om) \, \gamma_g^* \quad\text{for all}\;\; \om \in \widehat{K} \; , \; g \in \Gamma \; .
\end{equation}
We claim that for all $\om \in \widehat{K}$, $\Delta_*(U_\om)^* \, \al(U_\om) \in (\rL \Lambda)^\infty$. To prove this claim, fix $\om \in \widehat{K}$ and put $a := \Delta_*(U_\om)^* \, \al(U_\om)$. Define for every finite subset $\cF \subset \Gamma$, the von Neumann subalgebra $N_\cF \subset (\rL^\infty(X) \rtimes \Lambda)^\infty$ given by
$$N_\cF := (\rL^\infty(K^\cF) \rtimes \Lambda)^\infty \; .$$
Denote by $E_\cF$ the trace preserving conditional expectation of $(\rL^\infty(X) \rtimes \Lambda)^\infty$ onto $N_\cF$. Note that $(\rL \Lambda)^\infty \subset N_\cF$ for all $\cF \subset \Gamma$. Choose $\eps > 0$ and denote by $\|\,\cdot\,\|_2$ the $2$-norm on $(\rL^\infty(X) \rtimes \Lambda)^\infty$ given by the semi-finite trace $\Tr \ot \tau$. Since $\al(U_\om)$ commutes with $\rL^\infty(X/K)q$, it follows that $\al(U_\om) \in q(\rL^\infty(X) \rtimes \Lambda)^\infty q$.
So we can take a large enough finite subset $\cF \subset \Gamma$ such that
$$\|\al(U_\om) - E_\cF(\al(U_\om))\|_2 < \eps \quad\text{and}\quad \|q \Delta_*(U_\om) - E_\cF(q \Delta_*(U_\om))\|_2 < \eps \; .$$
It follows that for all $g \in \Gamma$, the element $\Delta_*(U_\om)^* \, \gamma_g \, \al(U_\om) \, \gamma_g^*$ lies at distance at most $2 \eps$ from $N_\cF$. By \eqref{eq.useful}, the element $a$ then lies at distance at most $2 \eps$ from $N_{\cF g^{-1}}$. Since $a$ also lies at distance at most $2 \eps$ from $N_\cF$, we conclude that $a$ lies at distance at most $4 \eps$ from $N_{\cF \cap  \cF g^{-1}}$ for all $g \in \Gamma$. We can choose $g$ such that $\cF \cap \cF g^{-1} = \emptyset$ and conclude that $a$ lies at distance at most $4 \eps$ from $(\rL \Lambda)^\infty$. Since $\eps > 0$ is arbitrary, the claim follows.

Put $V_\om := \Delta_*(U_\om)^* \, \al(U_\om)$ and $q_\om := V_\om V_\om^*$. Denote by $z \in \rL \Lambda$ the central support of $q$ in $(\rL \Lambda)^\infty$. Define
$$q_1 = \bigvee_{\om \in \widehat{K}} V_\om V_\om^* \; .$$
Since $V_\om \in (\rL \Lambda)^\infty$ and $V_\om^* V_\om = q$, it follows that $q_1 \in (\rL \Lambda)^\infty$ and $q_1 \leq z$. Since
$$V_\om V_\om^* = \Delta_*(U_\om)^* \, q \, \Delta_*(U_\om)$$
it also follows that $q_1$ commutes with all the unitaries $\Delta_*(U_\om)$, $\om \in \widehat{K}$. Being an element of $(\rL \Lambda)^\infty$, the projection $q_1$ certainly commutes with $\rL^\infty(X/K) = \Delta_*(\rL^\infty(X/K))$. Hence, $q_1$ commutes with the whole of $\Delta_*(\rL^\infty(X)) = \rL^\infty(X)$. So, $q_1 \in \B(\ell^2(\N)) \ot 1$. Since $q_1 \leq z$, it follows that $z = 1$. Because $(\Tr \ot \tau)(q) = \tau(p) \leq 1$, we must have $p = 1$ and find an element
$w \in \B(\ell^2(\N),\C) \ovt \rL \Lambda$ satisfying $w w^* = 1$ and $w^* w = q$. Replacing $\al$ by $\Ad w \circ \al$, we have found that $\al$ is an automorphism of $M$ satisfying
\begin{align*}
& \al(F) = \Delta_*(F) \;\;\text{for all $F \in \rL^\infty(X/K)$}\;\; , \;\; \al(U_\om) = \Delta_*(U_\om) \, V_\om \;\;\text{for all $\om \in \widehat{K}$, and} \\ & \al(u_g) = \gamma_g u_g \;\;\text{for all $g \in \Gamma$.}
\end{align*}
Here $V_\om \in \rL \Lambda$ are unitaries and $g \mapsto \gamma_g$ is a homomorphism from $\Gamma$ into the abelian group $\cU(\rL \Lambda)$. Since $[\Gamma,\Gamma] = \Gamma$, we actually have that $\gamma_g = 1$ for all $g \in \Gamma$. In particular $\al(\rL \Gamma) = \rL \Gamma$. Taking the relative commutant, it follows that $\al(\rL \Lambda) = \rL \Lambda$.

Since $(\Ad U_\om)(u_s) = \om(s) \, u_s$ for all $s \in \Lambda$, $\om \in \widehat{K}$, we also have that $(\Ad \al(U_\om))(\al(u_s)) = \om(s) \, \al(u_s)$ for all $s \in \Lambda$, $\om \in \widehat{K}$. Since $\al(U_\om) = \Delta_*(U_\om) \, V_\om$ and $V_\om$ belongs to the abelian algebra $\rL \Lambda$, we conclude that $(\Ad \Delta_*(U_\om))(\al(u_s)) = \om(s) \, \al(u_s)$ for all $s \in \Lambda$, $\om \in \widehat{K}$. On the other hand, since $\Delta$ is a $\delta$-conjugation for the action $K \actson X$, we also have that $(\Ad \Delta_*(U_\om))(u_k) = \om(\delta^{-1}(k)) \, u_k$ for all $k \in \Lambda$ and $\om \in \widehat{K}$. Writing the Fourier decomposition
$$\al(u_s) = \sum_{k \in \Lambda} \lambda^s_k \, u_k \quad\text{with}\;\; \lambda^s_k \in \C \; ,$$
we conclude that
$$\om(s) \, \lambda^s_k = \om(\delta^{-1}(k)) \, \lambda^s_k \quad\text{for all $s,k \in \Lambda$, $\om \in \widehat{K}$.}$$
So, for every $s \in \Lambda$ there is precisely one $k \in \Lambda$ with $\lambda^s_k \neq 0$ and this element $k \in \Lambda$ moreover satisfies $\om(s) = \om(\delta^{-1}(k))$ for all $\om \in \widehat{K}$. So, $k = \delta(s)$. In particular $\delta(s) \in \Lambda$ for all $s \in \Lambda$. So, $\delta(\Lambda) \subset \Lambda$. Since we can make a similar reasoning on $\al^{-1}(u_s)$, $s \in \Lambda$, it follows that $\delta(\Lambda) = \Lambda$.
\end{proof}

\section{A non commutative cocycle superrigidity theorem}

We prove the following twisted version of Popa's cocycle superrigidity theorem \cite[Theorem 0.1]{Po05} for Bernoulli actions of property (T) groups.

\begin{theorem}\label{thm.cocycle}
Let $K$ be a compact group with countable subgroup $\Lambda < K$. Let $\Gamma$ be a property (T) group. Put $X = K^\Gamma$ and denote by $\Lambda \actson X$ the action by diagonal translation. Put $N = \rL^\infty(X) \rtimes \Lambda$. Denote by $(\si_g)_{g \in \Gamma}$ the action of $\Gamma$ on $N$ such that $\si_g$ is the Bernoulli shift on $\rL^\infty(X)$ and the identity on $\rL \Lambda$. Let $p \in \rL \Lambda$ be a non zero projection.

\begin{itemize}
\item Assume that $q \in \B(\ell^2(\N)) \ovt \rL \Lambda$ is a projection, that $\gamma : \Gamma \recht \cU(q(\B(\ell^2(\N)) \ovt \rL \Lambda) q)$ is a group homomorphism and that $v \in \B(\C,\ell^2(\N)) \ovt N$ is a partial isometry satisfying $v^* v = p$ and $v v^* = q$. Then the formula
    $$\om_g := v^* \, \gamma_g \, \si_g(v)$$
    defines a $1$-cocycle for the action $(\si_g)_{g \in \Gamma}$ on $p N p$, i.e.\ a family of unitaries satisfying $\om_{gh} = \om_g \, \si_g(\om_h)$ for all $g,h \in \Gamma$.

\item Conversely, every $1$-cocycle for the action $(\si_g)_{g \in \Gamma}$ on $p N p$ is of the above form with $\gamma$ being uniquely determined up to unitary conjugacy in $\B(\ell^2(\N)) \ovt \rL \Lambda$.
\end{itemize}
\end{theorem}

Note that if $\Lambda$ is icc, we can take $q = p$, i.e.\ every $1$-cocycle for the action $(\si_g)_{g \in \Gamma}$ on $p N p$ is cohomologous with a homomorphism from $\Gamma$ to $\cU(p \rL (\Lambda)p)$. If $\Lambda$ is not icc, this is no longer true since $q$ and $e_{11} \ot p$ need not be equivalent projections in $\B(\ell^2(\N)) \ovt \rL \Lambda$, although they are equivalent in $\B(\ell^2(\N)) \ovt N$.

\begin{proof}
It is clear that the formulae in the theorem define $1$-cocycles. Conversely, let $p \in \rL \Lambda$ be a projection and assume that the unitaries $\om_g \in pNp$ define a $1$-cocycle for the action $(\si_g)_{g \in \Gamma}$ of $\Gamma$ on $pNp$. Consider the diagonal translation action $\Lambda \actson X \times K$ and put $\cN := \rL^\infty(X \times K) \rtimes \Lambda$. We embed $N \subset \cN$ by identifying the element $F u_\lambda \in N$ with the element $(F \ot 1)u_\lambda \in \cN$ whenever $F \in \rL^\infty(X), \lambda \in \Lambda$. Also $(\si_g)_{g \in \Gamma}$ extends naturally to a group of automorphisms of $\cN$ with $\si_g(1 \ot F) = 1 \ot F$ for all $F \in \rL^\infty(K)$. Define $P = \rL^\infty(K) \rtimes \Lambda$ and view $P$ as a subalgebra of $\cN$ by identifying the element $Fu_\lambda \in P$ with the element $(1 \ot F)u_\lambda \in \cN$ whenever $F \in \rL^\infty(K), \lambda \in \Lambda$ . Note that we obtained a commuting square
$$
\begin{matrix}
N & \subset & \cN\;\;\;\mbox{} \\
\cup & & \cup\;\;\;\mbox{} \\
\rL \Lambda & \subset & P\;\;.
\end{matrix}
$$
Define $\Delta : X \times K \recht X \times K : \Delta(x,k) = (k \cdot x, k)$ and denote by $\Delta_*$ the corresponding automorphism of $\rL^\infty(X \times K)$ given by $\Delta_*(F) = F \circ \Delta^{-1}$. One checks easily that the formula
$$\Psi : \rL^\infty(X) \ovt P \recht \cN : \Psi(F \ot G u_\lambda) = \Delta_*(F \ot G) u_\lambda \quad\text{for all}\;\; F \in \rL^\infty(X) , G \in \rL^\infty(K), \lambda \in \Lambda$$
defines a $*$-isomorphism satisfying $\Psi \circ (\si_g \ot \id) = \si_g \circ \Psi$ for all $g \in \Gamma$. Put $\mu_g := \Psi^{-1}(\om_g)$. It follows that $(\mu_g)_{g \in \Gamma}$ is a $1$-cocycle for the action $\Gamma \actson X$ with values in the Polish group $\cU(pPp)$. By Popa's cocycle superrigidity theorem \cite[Theorem 0.1]{Po05} and directly applying $\Psi$ again, we find a unitary $w \in \cU(p\cN p)$ and a group homomorphism $\rho : \Gamma \recht \cU(p P p)$ such that
$$\om_g = w^* \, \rho_g \, \si_g(w) \quad\text{for all}\;\; g \in \Gamma \; .$$
Consider the basic construction for the inclusion $N \subset \cN$ denoted by $\cN_1 := \langle \cN , e_N \rangle$. Put $T := w e_N w^*$. Since $\si_g(w) = \rho_g^* \, w \, \om_g$, it follows that $\si_g(T) = \rho_g^* \, T \, \rho_g$. Also note that $T \in \rL^2(\cN_1)$, that $T = p T$ and that $\Tr(T) = \tau(p)$.

Since we are dealing with a commuting square, we can identify the basic construction $P_1 := \langle P,e_{\rL \Lambda} \rangle$ for the inclusion $\rL \Lambda \subset P$ with the von Neumann subalgebra of $\cN_1$ generated by $P$ and $e_N$. In this way, one checks that there is a unique unitary
$$W : \rL^2(\cN_1) \recht \rL^2(X) \ot \rL^2(P_1) \quad\text{satisfying}\;\; W(ab) = a \ot b \quad\text{for all}\;\; a \in \rL^\infty(X), b \in \rL^2(P_1) \; .$$
The unitary $W$ satisfies $W(\rho_g \, \si_g(\xi) \, \rho_g^*) = (\si_g \ot \Ad \rho_g)(W(\xi))$ for all $g \in \Gamma$, $\xi \in \rL^2(\cN_1)$. The projection $T \in \rL^2(\cN_1)$ satisfies $T = \rho_g \, \si_g(T) \, \rho_g^*$. Since the unitary representation $(\si_g)_{g \in \Gamma}$ on $\rL^2(X) \ominus \C 1$ is a multiple of the regular representation, it follows that $W(T) \in 1 \ot \rL^2(P_1)$. This means that $T \in P_1$ and $\rho_g \, T = T \, \rho_g$ for all $g \in \Gamma$.

So we can view $T$ as the orthogonal projection of $\rL^2(P)$ onto a right $\rL \Lambda$ submodule of dimension $\tau(p)$ that is globally invariant under left multiplication by $\rho_g$, $g \in \Gamma$. Since $p T = T$, the image of $T$ is contained in $p \rL^2(P)$. Take projections $q_n \in \rL \Lambda$ with $\sum_n \tau(q_n) = \tau(p)$ and a right $\rL \Lambda$-linear isometry
$$\Theta : \bigoplus_{n \in \N} q_n \rL^2(\rL \Lambda) \recht \rL^2(P)$$
onto the image of $T$. Denote by $w_n \in \rL^2(P)$ the image under $\Theta$ of $q_n$ sitting in position $n$. Note that
$$w_n = p w_n \;\; , \;\; E_{\rL \Lambda}(w_n^* w_m) = \delta_{n,m} q_n \quad\text{and}\quad T = \sum_{n \in \N} w_n e_{\rL \Lambda} w_n^*$$
where in the last formula we view $T$ as an element of $P_1$. Identifying $P_1$ as above with the von Neumann subalgebra of $\cN_1$ generated by $P$ and $e_N$, it follows that
$$T = \sum_{n \in \N} w_n e_N w_n^* \; .$$
Since $T = w e_N w^*$, we get that $e_N = \sum_n w^* w_n e_N w_n^* w$. We conclude that $w^* w_n \in \rL^2(N)$ for all $n$. So $w_n^* w_m \in \rL^1(N)$ for all $n,m$. Since also $w_n^* w_m \in \rL^1(P)$, we have $w_n^* w_m \in \rL^1(\rL \Lambda)$. But then,
$$\delta_{n,m} q_n = E_{\rL \Lambda}(w_n^* w_m) = w_n^* w_m \; .$$
So, the elements $w_n$ are partial isometries in $P$ with mutually orthogonal left supports lying under $p$ and with right supports equal to $q_n$. Since $\sum_n \tau(q_n) = \tau(p)$, we conclude that the formula
$$W := \sum_n e_{1,n} \ot w_n$$
defines an element $W \in \B(\ell^2(\N),\C) \ovt P$ satisfying $WW^* = p$, $W^* W = q$ where $q \in (\rL \Lambda)^\infty = \B(\ell^2(\N)) \ovt \rL \Lambda$ is the projection given by $q := \sum_n e_{nn} \ot q_n$. We also have that $v := W^* w$ belongs to $\B(\C 1,\ell^2(\N)) \ovt N$ and satisfies $v^* v = p$, $v v^* = q$.

Finally, let $\gamma : \Gamma \recht \cU(q (\rL \Lambda)^\infty q)$ be the unique group homomorphism satisfying
$$\Theta(\gamma_g \xi) = \rho_g \Theta(\xi) \quad\text{for all}\;\; g \in \Gamma , \xi \in \bigoplus_n q_n \rL^2(\rL \Lambda) \; .$$
By construction $W \gamma_g = \rho_g W$ for all $g \in \Gamma$. Since $\om_g = w^* \, \rho_g \, \si_g(w)$ and since $W$ is invariant under $(\si_g)_{g \in \Gamma}$, we conclude that $\om_g = v^* \, \gamma_g \, \si_g(v)$.

We finally prove that $\gamma$ is unique up to unitary conjugacy. So assume that we also have $q_1,v_1$ and $\vphi : \Gamma \recht \cU(q_1 (\rL \Lambda)^\infty q_1)$ satisfying $\om_g = v_1^* \, \vphi_g \, \si_g(v_1)$. It follows that the element $v_1 v^* \in q_1 N^\infty q$ satisfies $\si_g(v_1 v^*) = \vphi_g^* \, v_1 v^* \, \gamma_g$ for all $g \in \Gamma$. As above this forces $v_1 v^*$ to belong to $q_1 (\rL \Lambda)^\infty q$, providing the required unitary conjugacy between $\gamma$ and $\vphi$.
\end{proof}

\section{\boldmath The classification of Cartan subalgebras up to automorphisms can be complete analytic; proof of Theorem \ref{thm.ex}} \label{sec.complete-analytic}

Let $M$ be a II$_1$ factor with separable predual. By \cite[Theorem 3 and its corollaries]{Ef64} the set of von Neumann subalgebras $A \subset M$ carries the standard Effros Borel structure. In this II$_1$ setting it can be described as the smallest $\sigma$-algebra such that for all $x,y \in M$ the map
\begin{equation}\label{eq.generating}
A \mapsto \tau(E_A(x)y)
\end{equation}
is measurable. Here $\tau$ denotes the unique tracial state on $M$ and $E_A$ is the unique trace preserving conditional expectation of $M$ onto $A$.

\begin{proposition}\label{prop.borel}
Let $M$ be a II$_1$ factor with separable predual. Equip the set of its von Neumann subalgebras with the Effros Borel structure, as explained above.
\begin{itemize}
\item The set $\cartan(M)$ of Cartan subalgebras of $M$ is a Borel set and as such a standard Borel space.
\item The equivalence relation of ``being unitarily conjugate'' is a Borel equivalence relation on $\cartan(M)$.
\item The equivalence relation of ``being conjugate by an automorphism of $M$'' is an analytic equivalence relation on $\cartan(M)$.
\end{itemize}
\end{proposition}

\begin{proof}
Denote by $\vnalg(M)$ the standard Borel space of von Neumann subalgebras of $M$.
By \cite[Theorem 3 and its corollaries]{Ef64} the map $A \mapsto A' \cap M$ is a Borel map on $\vnalg(M)$. Hence the set
$$\masa(M) = \{A \in \vnalg(M) \mid A' \cap M = A\}$$
of maximal abelian subalgebras of $M$ is a a Borel subset of $\vnalg(M)$.

Consider the Polish group $\cU(M)$ of unitaries in $M$ equipped with the $\|\, \cdot \,\|_2$ distance. Denote by $\subgr(\cU(M))$ the set of closed subgroups of $\cU(M)$ equipped with the Effros Borel structure. This Borel structure can be described as the smallest $\si$-algebra such that for all $u \in \cU(M)$ the map
$$\subgr(\cU(M)) \recht \R : \cG \mapsto d(u,\cG)$$
is measurable. Here and in what follows, $d(u,\cG)$ denotes the $\|\,\cdot\,\|_2$ distance of $u$ to $\cG$. We claim that the map
$$\masa(M) \recht \subgr(\cU(M)) : A \mapsto \cN_M(A) := \{u \in \cU(M) \mid u A u^* = A \}$$
is Borel. So we have to prove that for all $u \in \cU(M)$, the map $A \mapsto d(u,\cN_M(A))$ is Borel.
But \cite[Theorem 6.2 and Lemma 6.3]{PSS03} provides us with the following inequalities for every maximal abelian subalgebra $A \subset M$ and all $u \in \cU(M)$.
\begin{equation*}
\frac{1}{4} \|E_A - E_{uAu^*}\|_{\infty,2} \leq d(u,\cN_M(A)) \leq 90 \|E_A - E_{uAu^*}\|_{\infty,2} \; .
\end{equation*}
So in order to prove that $A \mapsto \cN_M(A)$ is Borel, it suffices to prove that for all $u \in \cU(M)$, the map $A \mapsto \|E_A - E_{uAu^*}\|_{\infty,2}$ is Borel. Choosing sequences $(x_n)$ and $(y_n)$ in $M$ such that $(x_n)$ is dense in the unit ball of $M$ while $(y_n)$ is dense in the unit ball of $\rL^2(M)$, we have that
$$\|E_A - E_{uAu^*}\|_{\infty,2} = \sup_{n,m} |\tau(E_A(x_n) y_m) - \tau(E_A(u^* x_n u) y_m)| \; .$$
Since the maps in \eqref{eq.generating} are Borel, also $A \mapsto \|E_A - E_{uAu^*}\|_{\infty,2}$ is Borel. This proves the claim that $A \mapsto \cN_M(A)$ is a Borel map from $\masa(M)$ to $\subgr(\cU(M))$.

Repeating the proof of \cite[Theorem 3]{Ef64} the map $\subgr(\cU(M)) \recht \vnalg(M) : \cG \mapsto \cG^{\prime\prime}$ is Borel. Together with the previous paragraph it follows that the map $\masa(M) \recht \vnalg(M) : A \mapsto \cN_M(A)^{\prime\prime}$ is Borel. In particular,
$$\cartan(M) = \{A \in \masa(M) \mid \cN_M(A)^{\prime\prime} = M\}$$
is a Borel subset of $\masa(M)$.

Consider the equivalence relation of ``being unitarily conjugate'' on the standard Borel space $\cartan(M)$. Let $(u_n)_{n \in \N}$ be a $\|\,\cdot\,\|_2$ dense sequence in $\cU(M)$. We claim that $A,B \in \cartan(M)$ are unitarily conjugate if and only if
\begin{equation}\label{eq.below12}
\inf_n \|E_{u_n A u_n^*} - E_B\|_{\infty,2} < \frac{1}{2} \; .
\end{equation}
Once this claim is proven, we observe as above that for every $n \in \N$ the map $(A,B) \mapsto \|E_{u_n A u_n^*} - E_B\|_{\infty,2}$ is Borel. It then follows that the unitarily conjugate Cartan subalgebras $(A,B)$ form a Borel subset of $\cartan(M) \times \cartan(M)$.

It remains to prove the claim. If $A$ and $B$ are unitarily conjugate, i.e.\ $u A u^* = B$, it suffices to take $n$ such that $\|u-u_n\|_2$ is small enough and \eqref{eq.below12} follows. Conversely, if \eqref{eq.below12} holds, take $n$ such that $\|E_{u_n A u_n^*} - E_B\|_{\infty,2} < 1/2$. In particular, for every $v \in \cU(u_n A u_n^*)$ we have that $\|E_B(v)\|_2 > 1/2$. Popa's conjugacy criterion \cite[Theorem A.1]{Po01} implies that $u_n A u_n^*$ and $B$ are unitarily conjugate. Hence also $A$ and $B$ are unitarily conjugate.

Finally, repeating the proof of \cite[Lemma 2.1]{Ef65} the map $\Aut(M) \times \cartan(M) \recht \cartan(M) : (\al,A) \mapsto \al(A)$ is seen to be Borel. Since the equivalence relation of ``being conjugate by an automorphism of $M$'' is precisely the orbit equivalence relation of this Borel action, it is an analytic equivalence relation.
\end{proof}

In \cite{DM07} it was shown that the isomorphism relation for torsion free abelian groups is complete analytic. This is done through the construction of a Borel reduction from the complete analytic problem of checking whether a tree has an infinite path. In this context a \emph{tree} is a set of finite sequences of natural numbers which is closed under taking initial segments. An \emph{infinite path} in a tree $T$ is an infinite sequence of natural numbers $x_1x_2\cdots$ such that $x_1\cdots x_n \in T$ for every $n \in \N$. A construction by Hjorth \cite{Hj01} associates to every tree $T$ a subgroup $G(T)$ of the countably infinite direct sum $\Lambda := \Q^{(\infty)}$ such that the isomorphism class of the torsion free abelian group $G(T)$ remembers whether or not $T$ admits an infinite path. We recall this construction in the following lemma and provide a second countable compactification $\Lambda \subset K$ such that $\overline{G(T)} \cap \Lambda = G(T)$ for every tree $T$.

Recall that $\Q/\Z = \bigoplus_{p \in \cP} \Z(p^\infty)$ where $\cP$ denotes the set of prime numbers and $\Z(p^\infty)=\Z[p^{-1}]/\Z$. Denote by $\rho_p$ the corresponding homomorphism $\rho_p: \Q \rightarrow \Z(p^\infty)$ uniquely characterized by $\rho_p(p^{-k}) = p^{-k} + \Z$ and $\rho_p(q^{-k}) = 0$ for all primes $q \neq p$ and all $k \in \N$.

Denote by $\N^{< \N}$ the set of finite sequences of natural numbers. For every $\si = (x_1\cdots x_n)$ denote by $|\si|=n$ the length of the sequence. Observe that the empty sequence $()$ belongs to $\N^{< \N}$ and that it is the unique element of length $0$. Recall that a tree $T$ is a subset $T \subset \N^{< \N}$ that is closed under taking initial segments: if $n \geq 1$ and $(x_1\cdots x_n) \in T$, then $(x_1\cdots x_k) \in T$ for all $k \leq n$. Whenever $\si \in \N^{< \N}$ with $|\si| \geq 1$ we denote by $\si^-$ the sequence obtained by removing the last element of $\si$. Whenever $\si \in \N^{< \N}$ and $b \in \N$ we denote by $\si b$ the sequence obtained by appending $b$ at the end of $\si$.

\begin{lemma}\label{lem.compactQ}
Choose for every prime number $p$ a second countable compactification $L_p$ of $\Z(p^\infty)$. Choose a second countable compactification $K_\Q$ of $\Q$ such that the homomorphisms $\rho_p : \Q \rightarrow L_p$ extend continuously to $K_{\Q}$ for all $p \in \cP$.

Consider the countable group $\Lambda$ given as the direct sum $\Lambda := \Q^{(\N^{< \N})}$.  For every $\si \in \N^{< \N}$ define the homomorphisms
$$\pi_\si: \Lambda \rightarrow \Q: x \mapsto x_\si \quad\text{and}\quad \gamma_\si: \Lambda \rightarrow \Q: x \mapsto x_\si + \sum_{b \in \N} x_{\si b} \; .$$
Define $K$ as the smallest compactification of $\Lambda$ such that the homomorphisms $\pi_\si: \Lambda \rightarrow K_{\Q}$ and $\gamma_\si: \Lambda \rightarrow K_{\Q}$ extend continuously to $K$ for all $\si \in \N^{< \N}$.

For every $\si \in \N^{< \N}$ we denote by $\delta_\si$ the obvious basis element of $\Lambda = \Q^{(\N^{<\N})}$. Let $p_0 < p_1 < \ldots$ be an enumeration of the prime numbers. Following \cite{Hj01} we define for every tree $T$ the subgroup $G(T)$ of $\Lambda$ generated by
$$\bigl\{ \; p_{2|\si|}^{-k} \, \delta_\si \; \big|\; \si \in T, k \in \N \;\bigr\} \;\; \cup \;\; \bigl\{ \; p_{2|\si|-1}^{-k} \, (\delta_\si - \delta_{\si^-}) \;\big|\; \si \in T, |\si|>0, k \in \N \; \bigr\} \; .$$
Denote by $\overline{G(T)}$ the closure of $G(T)$ inside $K$.

For every tree $T$ we have $\overline{G(T)} \cap \Lambda = G(T)$.
\end{lemma}

\begin{proof}
The definition of $G(T)$ and the continuity of the homomorphisms $\rho_p, \pi_\si$ and $\gamma_\si$ imply the following properties.
\begin{enumerate}
\item For all $\si \in \N^{< \N}$ with $\si \not\in T$ and for all $g \in \overline{G(T)}$, we have $\pi_\si(g) = 0$.
\item For all $\si \in T$, all prime numbers $p \notin \{p_{2|\si|-1}, p_{2|\si|}, p_{2|\si|+1}\}$ and all $g \in \overline{G(T)}$, we have $\rho_{p}(\pi_\si(g))=0$.
\item For all $\si \in T$ and all $g \in \overline{G(T)}$, we have $\rho_{p_{2|\si|+1}}(\gamma_\si(g))=0$.
\end{enumerate}

For all $g \in \Lambda$ we consider the finite set $S_g = \{(p,\si) \in \cP \times \N^{< \N} \mid \rho_p(\pi_\si(g)) \neq 0\}$. Note that $S_g = \emptyset$ if and only if $\pi_\si(g) \in \Z$ for every $\si \in \N^{< \N}$.

Assume that $g \in \overline{G(T)} \cap \Lambda$ with $g \not\in G(T)$. It follows that $S_g \neq \emptyset$. Indeed, otherwise we get that $\pi_\si(g) \in \Z$ for all $\si \in \N^{< \N}$. By property 1 above, we know that $\pi_\si(g) = 0$ for all $\si \not\in T$. Hence $g$ belongs to the group generated by $\{\delta_\si \mid \si \in T\}$. In particular, $g \in G(T)$.

Assume that $\overline{G(T)} \cap \Lambda$ is strictly larger than $G(T)$. Take $g \in (\overline{G(T)} \cap \Lambda) \setminus G(T)$ minimizing $|S_g|$. By the previous paragraph $S_g \neq \emptyset$. So we can choose $(p,\si) \in S_g$ maximizing $|\si|$. By property 1 above we have $\si \in T$. Property 2 implies that $p \in \{p_{2|\si|-1}, p_{2|\si|}, p_{2|\si|+1}\}$. We prove that every of these three possibilities leads to a contradiction.

First assume that $|\si|>0$ and $p = p_{2|\si|-1}$. By property 3 above we have $\rho_p(\gamma_{\si^-}(g))=0$. Using property 1 it follows that
$$\rho_p(g_{\si^-}) + \sum_{b \in \N, \si^- b \in T} \rho_p(g_{\si^- b})=0 \; .$$
So we can take $h$ in the group generated by $\{ p^{-k} \, (\delta_{\si^- b} - \delta_{\si^-}) \mid b \in \N, \si^- b \in T, k \in \N \}$ such that $\rho_p(h_{\si^-}) = \rho_p(g_{\si^-})$ and $\rho_p(h_{\si^- b}) = \rho_p(g_{\si^- b})$ for all $b \in \N$. Note that $h \in G(T)$ and hence $g-h \in (\overline{G(T)} \cap \Lambda) \setminus G(T)$. By construction
$$S_{g-h} = S_g \setminus \bigl(\; \{(p,\si^-)\} \; \cup \;  \{(p,\si^- b) \mid b \in \N\} \; \bigr) \; .$$
So $|S_{g-h}| < |S_g|$ contradicting the minimality of $|S_g|$.

Next assume that $p = p_{2|\si|}$. Choose $h$ in the group generated by $\{ p^{-k} \, \delta_\si \mid k \in \N \}$ such that $\rho_p(h_\si) = \rho_p(g_\si)$. Note that $h \in G(T)$ and hence $g-h \in (\overline{G(T)} \cap \Lambda) \setminus G(T)$. Again $|S_{g-h}| < |S_g|$ contradicting the minimality of $|S_g|$.

Finally assume that $p = p_{2|\si|+1}$. By property 3 above we have that $\rho_p(\gamma_{\si}(g))=0$. Hence
$$\rho_p(\pi_\si(g)) + \sum_{b \in \N} \rho_p(\pi_{\si b}(g)) = 0 \; .$$
As $\rho_p(\pi_\si(g)) \neq 0$, it follows that there exists a $b \in \N$ such that $\rho_p(\pi_{\si b}(g)) \neq 0$. This contradicts the maximality of $|\si|$.

In all three cases we obtained a contradiction. This ends the proof of the lemma.
\end{proof}

We are now ready to prove Theorem \ref{thm.ex}.

\begin{proof}[Proof of Theorem \ref{thm.ex}]
The first part of Theorem \ref{thm.ex} follows immediately from Theorem \ref{thm.main} and the following observations.
\begin{itemize}
\item The subgroups $K(\cP_1)$ and $K(\cP_2)$ are commensurate if and only if the symmetric difference $\cP_1 \bigtriangleup \cP_2$ is finite.
\item Considering the order of the elements in $K$, it follows that every automorphism $\delta \in \Aut(K)$ globally preserves each direct summand $\Z/p\Z$. In particular $\delta(\Lambda) = \Lambda$ and $\delta(\Lambda K(\cP_1)) = \Lambda K(\cP_1)$ for every subset $\cP_1 \subset \cP$.
\end{itemize}

It remains to prove the second part of Theorem \ref{thm.ex}. Put $\Lambda = \Q^{(\N^{< \N})}$ and denote by $K$ the compactification of $\Lambda$ given by Lemma \ref{lem.compactQ}. Consider the II$_1$ factor $M := \rL^\infty(K^\Gamma) \rtimes (\Gamma \times \Lambda)$.
For every tree $T \subset \N^{< \N}$ denote by $G(T) \subset \Lambda$ the subgroup defined in Lemma \ref{lem.compactQ}. Denote by $\cC(T) := \cC(\overline{G(T)})$ the associated Cartan subalgebra.

We denote the set of trees by $\trees$. Viewing $\trees$ as a subset of the power set of $\N^{< \N}$, it is a standard Borel space. It is straightforward though tedious to check that the map
$$\trees \recht \cartan(M) : T \mapsto \cC(T)$$
is Borel.

A classical result (see e.g.\ \cite[Theorem 27.1]{Ke95}) says that the subset of $\trees$ consisting of the trees that admit an infinite path, is complete analytic. In combination with \cite[Lemma 3.1]{DM07} we get a complete analytic set $X \subset \R$ and Borel functions $F_1,F_2: \R \rightarrow \trees$ such that the following statements hold.
\begin{itemize}
\item If $x \in X$, then $F_1(x) \cong F_2(x)$.
\item If $x \notin X$, then $F_1(x)$ has an infinite path and $F_2(x)$ has no infinite path.
\end{itemize}
Moreover we may assume that for all $x \in \R$ the trees $F_1(x)$ and $F_2(x)$ consist of sequences of even numbers only (for instance by ``multiplying by $2$'' the original $F_1,F_2$).

Define the Borel map
$$h : \R \recht \cartan(M) \times \cartan(M) : h(x) = (\cC(F_1(x)),\cC(F_2(x))) \; .$$
Denote by $\cR \subset \cartan(M) \times \cartan(M)$ the equivalence relation given by the Cartan subalgebras $(A,B)$ that are conjugate by an automorphism of $M$. In Proposition \ref{prop.borel} we have seen that $\cR$ is an analytic subset of $\cartan(M) \times \cartan(M)$. We prove below that $h^{-1}(\cR) = X$. Since $X$ is complete analytic, it then follows that $\cR$ is complete analytic.

The formula $h^{-1}(\cR) = X$ follows, once we have shown the following two statements.

\begin{itemize}
\item If $S, T \subset \N^{< \N}$ are trees that consist of sequences of even numbers and if $S \cong T$, then the Cartan subalgebras $\cC(S)$ and $\cC(T)$ are conjugate by an automorphism of $M$.
\item If $S, T \subset \N^{< \N}$ are trees such that $S$ admits an infinite path and such that the Cartan subalgebras $\cC(S)$ and $\cC(T)$ are conjugate by an automorphism of $M$, then $T$ also admits an infinite path.
\end{itemize}

To prove the first statement, let $\psi : S \recht T$ be an isomorphism of trees, i.e.\ a bijection that preserves initial segments. Since $S$ and $T$ only contain sequences of even numbers, $\psi$ can be extended, by induction on the length $|\sigma|$, to an isomorphism $\psi : \N^{< \N} \recht \N^{< \N}$. Denote by $\al \in \Aut(\Lambda)$ the corresponding automorphism given by $\al(\delta_\si) = \delta_{\psi(\sigma)}$ for all $\si \in \N^{< \N}$. Using the notation of Lemma \ref{lem.compactQ} we have that $\pi_{\psi(\si)} \circ \al = \pi_\si$ and $\gamma_{\psi(\si)} \circ \al = \gamma_\si$ for all $\si \in \N^{< \N}$. Hence $\al$ extends to a continuous automorphism of $K$. By construction $\al(G(S)) = G(T)$ and so $\al(\overline{G(S)}) = \overline{G(T)}$. It follows from Theorem \ref{thm.main} that the Cartan subalgebras $\cC(S)$ and $\cC(T)$ are conjugate by an automorphism of $M$.

To prove the second statement, assume that $S$ admits an infinite path and that the Cartan subalgebras $\cC(S)$ and $\cC(T)$ are conjugate by an automorphism of $M$. By Theorem \ref{thm.main} we get a continuous automorphism $\al$ of $K$ such that $\al(\Lambda) = \Lambda$ and $\al(\Lambda \overline{G(S)}) = \Lambda \overline{G(T)}$. It follows that $\al(\overline{G(S)})$ is commensurate with $\overline{G(T)}$ and that $\al(\overline{G(S)} \cap \Lambda)$ is commensurate with $\overline{G(T)} \cap \Lambda$. We conclude that $\al(G(S))$ and $G(T)$ are commensurate.

Let $(x_1 x_2 \cdots)$ be an infinite path in $S$. Define $g_0 = \delta_{()}$ and $g_i = \delta_{(x_1 \cdots x_i)}$. By construction, $g_i$ is a sequence of elements in $G(S)$ such that for all $i$ we have that $g_i$ is divisible by all the powers of $p_{2i}$ and that $g_{i+1} - g_i$ is divisible by all the powers of $p_{2i + 1}$. Denote by $N < \infty$ the index of $\al(G(S)) \cap G(T)$ in $G(T)$. Then $N \al(g_0), N \al(g_1), N \al(g_2) , \cdots$ is a sequence of elements in $G(T)$ satisfying the same divisibility conditions. It follows from \cite[Lemma 2.2]{DM07} that the tree $T$ admits an infinite path.
\end{proof}

\begin{remark}\label{rem.descriptive}
Theorem \ref{thm.ex} provides examples of II$_1$ factors $M$ such that the equivalence relation given by ``being conjugate by an automorphism'' on the set of Cartan subalgebras of $M$, is complete analytic. This is a mathematically rigorous way of saying that the decision problem whether two Cartan subalgebras of $M$ are conjugate by an automorphism, is as hard as it can possibly be. The decision problem whether two Cartan subalgebras of $M$ are unitarily conjugate is strictly easier. Indeed, Popa's unitary conjugacy criterion in \cite[Theorem A.1]{Po01} is rather concrete and, as we have seen in Proposition \ref{prop.borel}, shows that unitary conjugacy of Cartan subalgebras is always Borel. The fact that conjugacy by an automorphism can be complete analytic shows that there can never be a similar concrete criterion for conjugacy of Cartan subalgebras by an automorphism.

There are also other mathematically rigorous undecidability and non classifiability statements. It sounds plausible that there exist II$_1$ factors $M$ for which the Cartan subalgebras cannot be classified by countable structures. Entirely new techniques would however be needed to construct such II$_1$ factors. Indeed, the Cartan subalgebras, up to any kind of conjugacy, given by the admissible subgroups associated with a pair $\Lambda \subset K$, are by definition classifiable by countable structures.

As far as we know, it is still an open problem whether or not the isomorphism relation on countable abelian groups is universal among the equivalence relations that are classifiable by countable structures. If this problem can be solved in the affirmative, it is very likely that a technique similar to Theorem \ref{thm.ex} will provide II$_1$ factors $M$ for which the classification of Cartan subalgebras up to conjugacy by an automorphism, is in the same sense universal among equivalence relations that are classifiable by countable structures.
\end{remark}

\section{Factors with many group measure space Cartan subalgebras}

In this section we prove Theorem \ref{thm.gmsc}. We actually construct a concrete group $G$ and a concrete uncountable family $G \actson (X_\cF,\mu_\cF)$, indexed by all possible subsets $\cF \subset \N$, of free ergodic p.m.p.\ actions that are non stably orbit equivalent and that give rise to isomorphic crossed product II$_1$ factors.

The basic idea is due to \cite[Section 6.1]{Po06b} yielding the following concrete example of \emph{two} group actions $H \actson (X_i,\mu_i)$, $i = 1,2$, that are not orbit equivalent (although stably orbit equivalent) and yet have $\rL^\infty(X_1) \rtimes H \cong \rL^\infty(X_2) \rtimes H$. Let $H_0$ be a finite non commutative group. Define the infinite direct sum $H_1 := H_0^{(\N)}$ with compactification $K := H_0^\N$. Consider the action $H_1 \actson K$ by left translation. Define $X_1 := K^{\SL(3,\Z)}$ and consider the action of $H := \SL(3,\Z) \times H_1$ on $X_1$ where $\SL(3,\Z)$ acts by Bernoulli shift and $H_1$ acts diagonally. Take the first copy of $H_0$ in $H_1$ and put $X_2 := X_1 / H_0$. We have a natural action of $H \cong H/H_0$ on $X_2$. By construction $H \actson X_2$ is stably orbit equivalent (with compression constant $|H_0|^{-1}$) with $H \actson X_1$. Put $M = \rL^\infty(X_1) \rtimes H$. In  \cite[Section 6.1]{Po06b} it is shown that $M$ has fundamental group $\R_+$, while the orbit equivalence relation of $H \actson X_1$ has trivial fundamental group. Hence $\rL^\infty(X_1) \rtimes H \cong \rL^\infty(X_2) \rtimes H$ although $H \actson X_1$ and $H \actson X_2$ are not orbit equivalent.

We call two free ergodic p.m.p.\ actions \emph{W$^*$-equivalent} if their associated group measure space II$_1$ factors are isomorphic.

We now perform the following general construction. We start with two free ergodic p.m.p.\ actions $H \actson X_i$, $i=1,2$, that are not orbit equivalent but that are W$^*$-equivalent, and we construct a group $G$ with uncountably many non stably orbit equivalent but W$^*$-equivalent actions.

Assume that $\Gamma$ is a countable group and that $\Gamma_n < \Gamma$ is a sequence of subgroups of infinite index. Define the disjoint union $I := \bigsqcup_{n \in \N} \Gamma/\Gamma_n$ with the natural action $\Gamma \actson I$. To avoid trivialities, assume that $\Gamma \actson I$ is faithful, i.e.\ the intersection of all $g \Gamma_n g^{-1}$, $g \in \Gamma$, $n \in \N$, reduces to $\{e\}$. Define the generalized wreath product group
$$G := H \wr_I \Gamma = H^{(I)} \rtimes \Gamma \; .$$
Whenever $\cF \subset \N$, define the space
$$X_\cF := \prod_{n \in \cF} X_1^{\Gamma/\Gamma_n} \; \times \; \prod_{n \not\in \cF} X_2^{\Gamma / \Gamma_n} \; . $$
Taking an infinite product of copies of the action $H \actson X_i$, we get an action $H^{(\Gamma/\Gamma_n)} \actson X_i^{\Gamma/\Gamma_n}$. Taking a further infinite product of these actions over $n \in \cF$ (with $i=1$) and $n \in \N - \cF$ (with $i=2$), we obtain an action of $H^{(I)}$ on $X_\cF$. This action is compatible with the Bernoulli shift of $\Gamma$ on $X_\cF$. So we have constructed a free ergodic p.m.p.\ action
$$G \overset{\si_\cF}{\actson} X_\cF \; .$$
Denoting $M := \rL^\infty(X_1) \rtimes H \cong \rL^\infty(X_2) \rtimes H$, it is easily checked that
$$\rL^\infty(X_\cF) \rtimes G \cong \bigl( \underset{i \in I}{\ovt} M \bigr) \rtimes \Gamma \quad\text{where $\Gamma$ acts on $\underset{i \in I}{\ovt} M$ by shifting the indices.}$$
Hence all the actions $G \actson X_\cF$ are W$^*$-equivalent.

Theorem \ref{thm.gmsc} is an immediate consequence of the following more concrete result, showing that under the right assumptions, the actions $\si_\cF$ are mutually non stably orbit equivalent. We did not try to formulate the most general conditions under which such a result holds, but provide a concrete example of $\Gamma_n < \Gamma$ for which the proof is rather straightforward.

\begin{proposition} \label{prop.concrete}
Consider the group actions $H \actson X_i$, $i = 1,2$, defined above: $H \actson X_i$ are W$^*$-equivalent, non orbit equivalent and $H$ is a residually finite group. Define the field
$$K := \Q(\sqrt{n} \mid n \in \N)$$
by adding all square roots of positive integers to $\Q$. For all $n \geq 2$ we put $K_n := \Q(\sqrt{n})$. Define the countable group $\Gamma := \PSL(4,K)$ and the sequence of subgroups
$$\Gamma_n := \{a A \mid a \in K \; , \; A \;\text{is an upper triangular matrix with entries in $K_n$} \; , \; \det(a A) = 1 \; \} \; .$$
As above we put $I = \bigsqcup_{n \geq 2} \Gamma / \Gamma_n$ and $G = H \wr_I \Gamma$. For every subset $\cF \subset \N - \{0,1\}$, we are given the action $\si_\cF$ of $G$ on $X_\cF$.

The actions $\si_{\cF}$ and $\si_{\cF'}$ are stably orbit equivalent iff $\cF = \cF'$.
\end{proposition}

We freely make use of the following properties of $\Gamma_n < \Gamma$.
\begin{itemize}
\item $\Gamma$ is a simple group. In particular, $\Gamma$ has no non trivial amenable quotients.
\item The subgroups $\Gamma_n$ are amenable.
\item If $n \geq 2$ and $g \in \Gamma-\Gamma_n$, then $g \Gamma_n g^{-1} \cap \Gamma_n$ has infinite index in $\Gamma_n$, i.e.\ the quasi-normalizer of $\Gamma_n$ inside $\Gamma$ equals $\Gamma_n$. Equivalently: $\Gamma_n$ acts with infinite orbits on $\Gamma/ \Gamma_n - \{e \Gamma_n\}$.
\item If $n \neq m$ and $\delta \in \Aut \Gamma$, then $\delta(\Gamma_n) \cap \Gamma_m$ has infinite index in $\Gamma_m$. This follows because the automorphism group of $\Gamma$ is generated by the inner automorphisms, the automorphisms given by automorphisms of the field $K$ and the automorphism transpose-inverse, and because every field automorphism of $K$ globally preserves every $K_n$.
\item The previous two items imply that $(\Stab i) \cdot j$ is infinite for all $i \neq j$.
\end{itemize}

We also use the following lemma. It is a baby version of similar von Neumann algebra results (see \cite[Theorem 4.1]{Po03}, \cite[Theorem 0.1]{Io06} and \cite[Theorem 4.2]{IPV10}). Our lemma is also closely related to \cite[Theorem 6.7]{CSV09}. We include a short and elementary proof for the convenience of the reader.

\begin{lemma} \label{lemma.intertwine}
Let $\Gamma \actson I$ be any action of a countable group $\Gamma$ by permutations of a countable set $I$. Let $H$ be any countable group and put $G := H \wr_I \Gamma$. Assume that $\Lambda < G$ is a subgroup with the relative property (T). Then either a finite index subgroup of $\Lambda$ is contained in $H \wr_I \Stab i$ for some $i \in I$, or there exists a $g \in G$ such that $g \Lambda g^{-1} \subset \Gamma$.
\end{lemma}
\begin{proof}
Define the unitary representation $\pi : G \recht \cU(\ell^2(H \times I))$ given by
$$(\pi(h) \xi)(k,i) = \xi(h_i^{-1} k , i) \;\;\text{if $h \in H^{(I)}$, and}\quad (\pi(\gamma) \xi)(k,i) = \xi(k,\gamma^{-1} \cdot i) \;\;\text{if $\gamma \in \Gamma$.}$$
Define the function
$$\eta : H \times I \recht \C : \eta(k,i) = \begin{cases} 1 &\;\;\text{if $k=e$,} \\ 0 &\;\;\text{if $k \neq e$.}\end{cases}$$
The formal expression $b(g) := \pi(g) \eta - \eta$ provides a well defined $1$-cocycle of $G$ with values in $\ell^2(H \times I)$.
Since $\Lambda < G$ has the relative property (T), the $1$-cocycle $b$ is inner on $\Lambda$. So we find $\xi \in \ell^2(H \times I)$ satisfying
$$b(\lambda) = \pi(\lambda) \xi - \xi \quad\text{for all}\;\; \lambda \in \Lambda \; .$$
Denote by $\delta_h : H \recht \C$ the function that equals $1$ in $h$ and $0$ elsewhere. We get that
$$\xi(k,i)  = \xi(h_i^{-1} k , \gamma^{-1} \cdot i) - \delta_e(h_i^{-1} k) + \delta_e(k) \quad\text{for all}\;\; (h,\gamma) \in \Lambda \; , \; k \in H \; , \; i \in I \; .$$
For every $i \in I$, we write $\xi_i \in \ell^2(H)$ given by $\xi_i(k) = \xi(k,i)$. We also denote by $(\lambda_h)_{h \in H}$ the left regular representation of $H$ on $\ell^2(H)$. The above formula becomes
\begin{equation}\label{eq.almost}
\xi_i = \lambda_{h_i} \xi_{\gamma^{-1} \cdot i} - \delta_{h_i} + \delta_e \quad\text{for all}\;\; (h,\gamma) \in \Lambda \; , \; i \in I \; .
\end{equation}
Since $\xi \in \ell^2(H \times I)$, it follows that $\|\xi_i\| \recht 0$ as $i \recht \infty$. Assume that finite index subgroups of $\Lambda$ are never contained in $H \wr_I \Stab i$. We then find a sequence $(h_n,\gamma_n) \in \Lambda$ such that  for all $i \in I$ we have that $\gamma_n^{-1} \cdot i \recht \infty$ as $n \recht \infty$. Denote by $V \subset \ell^2(H)$ the subset of vectors $V := \{\delta_e - \delta_h \mid h \in H\}$. Note that $V \subset \ell^2(\N)$ is closed. We apply \eqref{eq.almost} to $(h_n,\gamma_n) \in \Lambda$ and let $n \recht \infty$. Since $\|\xi_{\gamma_n^{-1} \cdot i}\| \recht 0$, we conclude that $\xi_i \in V$ for every $i \in I$. Denote by $b_i \in H$ the unique element such that $\xi_i = \delta_{e} - \delta_{b_i}$. Since $\|\xi_i\| \recht 0$ as $i \recht \infty$, we conclude that $b_i = e$ for all but finitely many $i \in I$. So $b := (b_i)_{i \in I}$ is a well defined element of $H^{(I)}$. Formula \eqref{eq.almost} becomes
$$b_i^{-1} \, h_i \, b_{\gamma^{-1} \cdot i} = e \quad\text{for all}\;\; (h,\gamma) \in \Lambda \; , \; i \in I \; ,$$
which precisely means that $b^{-1} \Lambda b \subset \Gamma$.
\end{proof}

We can now prove Proposition \ref{prop.concrete}. We use the terminology and conventions concerning stable orbit equivalences that we introduced after Remark \ref{rem.McDuff}.

\begin{proof}
Let $\cF,\cF' \subset \N - \{0,1\}$ be subsets. Assume that $\Delta : X_{\cF} \recht X_{\cF'}$ is a stable orbit equivalence with corresponding Zimmer $1$-cocycle
$$\om : G \times X_{\cF} \recht G \quad , \quad \Delta(g \cdot x) = \om(g,x) \cdot \Delta(x) \;\;\text{for all $g \in G$ and a.e.\ $x \in X_{\cF}$.}$$
Since both $X_1$ and $X_2$ are standard non atomic probability spaces, the restriction of the action $G \overset{\si_\cF}{\actson} X_\cF$ to $\Gamma$ is the generalized Bernoulli action $\Gamma \actson [0,1]^I$. Note that its further restriction to the property (T) group $\Lambda := \PSL(4,\Z) < \Gamma$ is weakly mixing, since $\Lambda \cdot i$ is infinite for every $i \in I$.
By Popa's cocycle superrigidity theorem \cite[Theorem 0.1]{Po05} we can replace $\Delta$ by a similar stable orbit equivalence and assume that $\om(\lambda,x) = \delta(\lambda)$ for all $\lambda \in \Lambda$ and a.e.\ $x \in X_\cF$, where $\delta : \Lambda \recht G$ is a group homomorphism.

For all $i,j \in \{1,2,3,4\}$ denote by $e_{ij}$ the matrix having $1$ in position $ij$ and $0$ elsewhere. Define, for $i \neq j$ and $k \in K$, the elementary matrix $E_{ij}(k) = 1 + k e_{ij}$. The elementary subgroups $E_{ij}(K)$ generate $\Gamma$. Given distinct $i,j \in \{1,2,3,4\}$, denote by $i',j'$ the remaining elements of $\{1,2,3,4\}$ and denote by $\Lambda_{ij}$ the subgroup of $\Lambda$ generated by $E_{i'j'}(\Z)$ and $E_{j'i'}(\Z)$. Note that $\Lambda_{ij}$ is non amenable and that $\Lambda_{ij}$ commutes with $E_{ij}(K)$. Since $\Lambda_{ij} \actson X_{\cF}$ is weakly mixing and $\om$ is a homomorphism on $\Lambda_{ij}$, \cite[Proposition 3.6]{Po05} implies that $\om$ is a homomorphism on $E_{ij}(K)$. Since this holds for all $i\neq j$, we conclude that $\om$ is a homomorphism on the whole of $\Gamma$. Denote this homomorphism by $\delta : \Gamma \recht G$.

Since $\Gamma$ is simple, $\delta$ is injective. Recall that $G = H \wr_I \Gamma$. We shall prove that there exists a $g \in G$ such that $g \delta(\Gamma) g^{-1} \subset \Gamma$. Denote by $\Gamma_{12}$ the copy of $\PSL(2,K)$ inside $\Gamma$ generated by $E_{12}(K)$ and $E_{21}(K)$. We first claim that there is no $i \in I$ such that $\delta(\Gamma_{12}) \subset H \wr_I \Stab i$. Assume the contrary. Denote by $\pi : G \recht \Gamma$ the natural quotient map. Since $\Gamma_{12}$ is simple and non amenable, $\pi(\delta(\Gamma_{12})) \subset \Stab i$ must be the trivial group. So, $\delta(\Gamma_{12}) \subset H^{(I)}$. Since $H$ is residually finite and $\delta$ is injective, this is absurd. This proves the claim. Since $\Gamma_{12}$ has no non trivial finite index subgroups, we only retain the existence of $\gamma_n \in \Gamma_{12}$ such that for all $i \in I$ we have that $\pi(\delta(\gamma_n)) \cdot i \recht \infty$ as $n \recht \infty$.

We next claim that there is no finite index subgroup $\Lambda_0 < \Lambda$ and $i \in I$ with $\delta(\Lambda_0) \subset H \wr_I \Stab i$. To prove this claim, assume the contrary. Since $\Stab i$ is amenable and $\Lambda_0$ has property (T), it follows that $\pi(\delta(\Lambda_0))$ is finite. So making $\Lambda_0$ smaller but still of finite index, we get that $\delta(\Lambda_0) \subset H^{(I)}$. Since $\Lambda_0$ is finitely generated, we can take $I_1 \subset I$ finite such that $\delta(\Lambda_0) \subset H^{(I_1)}$. Recall the subgroup $\Lambda_{12} < \Lambda$ that commutes with the subgroup $\Gamma_{12} < \Gamma$ that we considered in the previous paragraph. Also recall that we have found elements $\gamma_n \in \Gamma_{12}$ such that for all $i \in I$ we have that $\pi(\delta(\gamma_n)) \cdot i \recht \infty$ as $n \recht \infty$. Since $\delta(\Lambda_0 \cap \Lambda_{12}) \subset H^{(I_1)}$, we have
$$\delta(\gamma_n) \, \delta(\Lambda_0 \cap \Lambda_{12}) \, \delta(\gamma_n)^{-1} \in H^{(\pi(\delta(\gamma_n)) \cdot I_1)} \; .$$
Since $\gamma_n$ and $\Lambda_{12}$ commute, the left hand side equals $\delta(\Lambda_0 \cap \Lambda_{12})$. Letting $n \recht \infty$, the sets $\pi(\delta(\gamma_n)) \cdot I_1$ and $I_1$ are eventually disjoint. We conclude that $\delta(\Lambda_0 \cap \Lambda_{12}) = \{e\}$. Since $\Lambda_0 < \Lambda$ has finite index, the group $\Lambda_0 \cap \Lambda_{12}$ is infinite, contradicting the injectivity of $\delta$.

Combining the claim that we proved in the previous paragraph with Lemma \ref{lemma.intertwine}, we can conjugate $\delta$, replace $\Delta$ by a similar stable orbit equivalence and assume that $\delta(\Lambda) \subset \Gamma$. For all $i,j$ we know that $\delta(\Lambda_{ij})$ is a non amenable subgroup of $\Gamma$. Therefore $\delta(\Lambda_{ij}) \cdot k$ is infinite for all $k \in I$ and the centralizer of $\delta(\Lambda_{ij})$ inside $G$ is a subgroup of $\Gamma$. Therefore $\delta(E_{ij}(K)) \subset \Gamma$ for all $i,j$. Hence, $\delta(\Gamma) \subset \Gamma$.

We can perform a similar reasoning on the generalized inverse stable orbit equivalence $\Deltab : X_{\cF'} \recht X_{\cF}$ and may assume that $\Deltab(g \cdot x) = \deltab(g) \cdot \Deltab(x)$ for all $g \in \Gamma$ and a.e.\ $x \in X_{\cF'}$, where $\deltab : \Gamma \recht \Gamma$ is an injective group homomorphism. By definition $\Deltab(\Delta(x)) \in G \cdot x$ for a.e.\ $x \in X_{\cF}$. So we find a measurable function $\vphi : X_\cF \recht G$ such that $\Deltab(\Delta(x)) = \vphi(x) \cdot x$ for a.e.\ $x \in X_\cF$. It follows that
$$\vphi(g \cdot x) = \deltab(\delta(g)) \, \vphi(x) \, g^{-1} \quad\text{for all}\;\; g \in \Gamma \;\;\text{and a.e.}\;\; x \in X_\cF \; .$$
So the non negligible subset $\{(x,x') \in X_\cF \times X_\cF \mid \vphi(x) = \vphi(x')\}$ is invariant under the diagonal $\Gamma$-action. Since $\Gamma \actson X_\cF$ is weakly mixing, we conclude that $\vphi$ is essentially constant. This constant value necessarily belongs to $\Gamma$ and we conclude that $\deltab \circ \delta$ is an inner automorphism of $\Gamma$. A similar reasoning holds for $\delta \circ \deltab$. In particular $\Delta(\Gamma \cdot x) = \Gamma \cdot \Delta(x)$ for a.e.\ $x \in X_\cF$. Lemma \ref{lem.elem} implies that $\Delta$ is a measure space isomorphism of $X_\cF$ onto $X_{\cF'}$ and that $\Delta$ is a $\delta$-conjugacy of the respective $\Gamma$-actions. In particular, the compression constant of $\Delta$ equals $1$.

Denote
$$X^i_\cF := \begin{cases} X_1 &\;\;\text{if}\;\; i \in \Gamma / \Gamma_n \;\;\text{and}\;\; n \in \cF \; , \\
X_2 &\;\;\text{if}\;\; i \in \Gamma / \Gamma_n \;\;\text{and}\;\; n \not\in \cF \; .\end{cases}$$
By construction, $X_\cF = \prod_{i \in I} X^i_\cF$ and similarly for $X_{\cF'}$.

By \cite[Lemma 6.15]{PV09} the $\delta$-conjugacy $\Delta$ necessarily decomposes as a product of measure space isomorphisms. More precisely,
we find a permutation $\eta$ of the set $I$ and measure space isomorphisms $\Delta_i : X^i_\cF \recht X^{\eta(i)}_{\cF'}$ such that
$$\eta(g \cdot i) = \delta(g) \cdot \eta(i) \quad\text{and}\quad \Delta(x)_{\eta(i)} = \Delta_i(x_i) \quad\text{for all $g \in \Gamma$ and a.e. $x \in X_\cF$.}$$
Since $\eta$ preserves the orbits of $\Gamma \actson I$, it defines a permutation $\etab$ of $\N$ such that $\eta(i) \in \Gamma/\Gamma_{\etab(n)}$ iff $i \in \Gamma/\Gamma_n$. Also $\Delta_i$ only depends on the orbit of $i$, yielding $\Delta_n : X^i_\cF \recht X^{\eta(i)}_{\cF'}$ whenever $i \in \Gamma/\Gamma_n$.

Note that $\delta(\Stab i) = \Stab \eta(i)$. For all $n \neq m$ and all $\delta \in \Aut \Gamma$ we have that $\delta(\Gamma_n) \neq \Gamma_m$. Hence $\etab$ is the identity. Denote by $H^i$ the copy of $H$ in $H^{(I)}$ sitting in position $i$. Fix $n \in \N$ and $i \in \Gamma/\Gamma_n$. Put $j := \eta(i)$ and note that $j \in \Gamma/\Gamma_n$. Since the quasi-normalizer of $\Gamma_n$ inside $\Gamma$ equals $\Gamma_n$ and since $\Gamma_n$ acts with infinite orbits on $\Gamma/\Gamma_m$ for all $m \neq n$, it follows that the action
$$\Stab i \actson \prod_{j, j \neq i} X^j_\cF$$
is weakly mixing. Since the $1$-cocycle $\om$ is a homomorphism on $\Stab i$ and $H^i$ commutes with $\Stab i$, it follows from \cite[Proposition 3.6]{Po05} that for all $h \in H^i$, the map $x \mapsto \om(h,x)$ only depends on the coordinate $x_i$ and that $\om(h,x)$ commutes with $\delta(\Stab i) = \Stab j$. Since $(\Stab j) \cdot k$ is infinite for all $k \neq j$ and since $\Stab j < \Gamma$ has a trivial centralizer in $\Gamma$, it follows that $\om(h,x) \in H^j$. A similar reasoning holds for the inverse $1$-cocycle and we conclude that $\Delta_n : X^i_\cF \recht X^j_{\cF'}$ is an orbit equivalence between $H \actson X^i_\cF$ and $H \actson X^j_{\cF'}$. Since both $i,j \in \Gamma/\Gamma_n$ and since the actions $H \actson X_1$ and $H \actson X_2$ are not orbit equivalent, it follows that either $n$ belongs to both $\cF$ and $\cF'$ or that $n$ belongs to both $\N - \cF$ and $\N - \cF'$. This holds for all $n$ and we conclude that $\cF = \cF'$.
\end{proof}

\end{document}